\documentclass[11pt,reqno]{amsart}
\usepackage{graphicx,enumerate,amssymb,bbold}
\usepackage[notrig]{physics}
\usepackage[ruled,vlined]{algorithm2e}
\usepackage[left=2.5cm,right=2.5cm,top=2.5cm,bottom=2.5cm,includeheadfoot]{geometry} 

\usepackage[colorlinks=true, pdfstartview=FitV, linkcolor=blue, 
            citecolor=blue, urlcolor=blue]{hyperref}

\numberwithin{equation}{section}
\numberwithin{figure}{section}
\theoremstyle{plain}
\newtheorem{theorem}{Theorem}[section]
\newtheorem{lemma}[theorem]{Lemma}
\newtheorem{proposition}[theorem]{Proposition}

\theoremstyle{definition}
\newtheorem{definition}{Definition}[section]

\newtheorem{remark}{Remark}[section]

\newcommand{\bitem}{\begin{itemize}}
\newcommand{\eitem}{\end{itemize}}
\newcommand{\mc}[1]{\mathcal{#1}}

\newcommand{\N}{\mathbb{N}}
\newcommand{\R}{\mathbb{R}}

\newcommand{\bpm}{\begin{pmatrix}}
\newcommand{\epm}{\end{pmatrix}}
\newcommand{\bsm}{\left(\begin{smallmatrix}}
\newcommand{\esm}{\end{smallmatrix}\right)}
\newcommand{\T}{\top}
\newcommand{\ol}[1]{\overline{#1}}

\newcommand{\wt}[1]{\widetilde{#1}}
\newcommand{\la}{\langle}
\newcommand{\ra}{\rangle}

\newcommand{\mrm}[1]{\mathrm{#1}}

\newcommand{\veps}{\varepsilon}

\newcommand{\w}{\omega}

\newcommand{\vphi}{\varphi}

\newcommand{\eins}{\mathbb{1}}

\DeclareMathOperator{\Diag}{Diag}

\DeclareMathOperator{\dom}{dom}
\DeclareMathOperator{\intr}{int}

\DeclareMathOperator{\argmin}{arg min}

\DeclareMathOperator{\supp}{supp}

\DeclareMathOperator{\ggrad}{grad}

\DeclareMathOperator{\KL}{KL}

\newcommand\numberthis{\addtocounter{equation}{1}\tag{\theequation}}          

\makeatletter
\renewcommand{\@maketitle@hook}{
    \parbox[c]{.5\linewidth}{Pure and Applied Functional Analysis\\Volume 8, Number 3, 2023, 855--880}\par
    \vspace{15pt}
}
\makeatother

\title[Multilevel Geometric Optimization]{Multilevel Geometric Optimization\\for Regularised Constrained Linear Inverse Problems}

\author[S.~M\"uller]{Sebastian M\"uller}
\address[S.~M\"uller]{Dept.~Mathematics and Comp.~Science, Heidelberg University, Germany} 
\email{sebastian.mueller@math.uni-heidelberg.de}

\author[S.~Petra]{Stefania Petra}
\address[S.~Petra]{Dept.~Mathematics and Comp.~Science, Heidelberg University, Germany} 
\email{petra@math.uni-heidelberg.de}
\urladdr{\url{https://stpetra.com}}

\author[M.~Zisler]{Matthias Zisler}
\address[M.~Zisler]{Dept.~Mathematics and Comp.~Science, Heidelberg University, Germany} 
\email{zisler@math.uni-heidelberg.de}

\keywords{multilevel geometric optimization, nonlinear prolongation and restriction, information geometry, Riemannian optimization}
\subjclass[2010]{65K10, 49J40, 49M37, 68U10, 74P20, 90C06}
\dedicatory{Dedicated to Professor Lev Bregman on occasion of his 80th birthday}

\begin{document}
\begin{abstract}
We present a geometric multilevel optimization approach that smoothly incorporates box constraints. Given a box constrained optimization problem, we consider a hierarchy of models with varying discretization levels. Finer models are accurate but expensive to compute, while coarser models are less accurate but cheaper to compute. When working at the fine level, multilevel optimisation computes the search direction based on a coarser model which speeds up updates at the fine level. Moreover, exploiting geometry induced by the hierarchy the feasibility of the updates is preserved. In particular, our approach extends classical components of multigrid methods like restriction and prolongation to the Riemannian structure of our constraints.
\end{abstract}
\maketitle
\tableofcontents

\section{Introduction}\label{sec:Introduction}
\subsection{Overview, Motivation}
Multiscale representations have a long tradition in numerical analysis \cite{Hackbusch:1985,Trottenberg:2001} and in signal and image processing \cite{Mallat:2009aa,Christensen:2015wp}, but much less so in numerical optimization if we ignore unconstrained convex quadratic optimization problems that boil down to solving a linear symmetric positive-definite system. 
In this paper, we consider constrained convex optimization problems of the form
\begin{equation}\label{eq:problem-basic}
    \min_{y\in [0,1]^{n}} f(y)
\end{equation}
with smooth objective functions $f\colon \R^{n}\to\R$. The main goal is to accelerate a first-order iterative algorithm for solving \eqref{eq:problem-basic} by adopting a multilevel problem representation analogous to numerical multigrid, in order to compute efficiently descent directions for $f$ at the original fine level using a much smaller number set of variables at coarser levels. 

A key ingredient of the approach is to replace the feasible set $[0,1]^{n}$ by the interior $(0,1)^{n}$ that is turned into a Riemannian manifold $(\mc{B}^{n},g)$ using information geometry \cite{Amari:2000aa}. This geometry differs from the more common geometry underlying interior point methods \cite{Nesterov:2002aa} (see Section \ref{sec:related-work}). Transfer of variables from fine to coarse levels and vice versa are performed by restriction and prolongation mappings, that are adapted to the geometry and whose differentials are used for pushing tangent vectors between levels. Criteria for invoking coarse level computation and certifying resulting search directions are provided. Numerical results illustrate the effectiveness of our approach for the specific instance of \eqref{eq:problem-basic}
\begin{equation}\label{eq:intro-approach}
\min_{y\in[0,1]^n} f(y),\qquad f(y):= D_\varphi(Ay,b) + J(y),
\end{equation}
where $b\in\R_{++}^{p}$ is a positive measurement vector, $A \in \R^{p\times n}$ is a nonnegativity-preserving linear mapping, $J$ is a smoothed total variation regularizer \cite[eq.~(2.5)]{Censor:2020aa} and $D_{\vphi}$ is the Bregman divergence with respect to the convex function 
\begin{equation}\label{eq:def-Breg-inducing-func}
\varphi \colon \R^{n} \to \R,\qquad
y \mapsto \varphi(y) = \la y, \log y \ra - \la \eins, y \ra.
\end{equation}
See Figure \ref{fig:find-coarse-A} and the caption for further explanation.

Although we consider the specific problem \eqref{eq:problem-basic} in this paper, we believe that our approach generalizes to other constrained convex programs, analogous to the way how open convex parameter sets of probability distributions are turned into statistical manifolds \cite{Lauritzen:1987aa,Amari:2000aa}.

\begin{figure}
    \centerline{
    \includegraphics[width=0.3\textwidth]{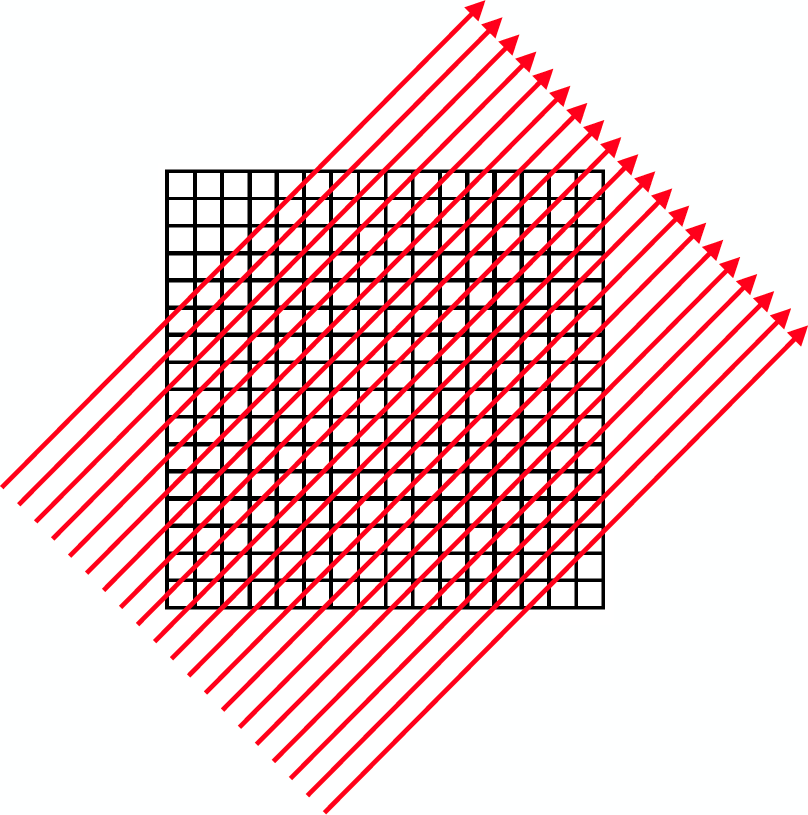}
    \hspace{0.1\textwidth}
    \includegraphics[width=0.3\textwidth]{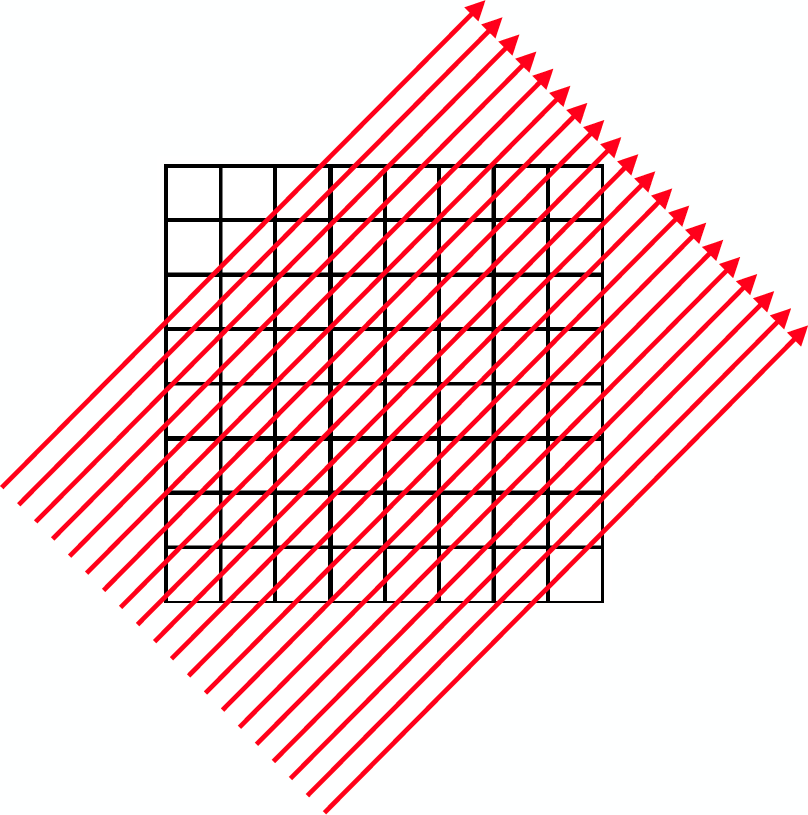}
    }
    \caption{\textbf{Left.} Incidence geometry of projection rays and cells indexed by $I_{n}$, that cover a subset in $\R^{d}$ (here: $d=2$ for illustration). The corresponding line integrals define the matrix $A$ that linearly maps cell values $y_{i},\, i\in I_{n}$ to projections $b_{j},\, j\in \{1,\dotsc,p\}$. This gives rise to an underdetermined linear system $A y = b$. The task is to recover $y$ from $b$ by solving \eqref{eq:intro-approach} using the constraints and the regularizer to obtain a well-posed problem. \textbf{Right.} Reducing the spatial resolution yields a coarse level representation of the projection matrix, 
    that enables to evaluate a surrogate of the objective function on the coarse level which does not involve the data vector $b$ recorded on the fine level. 
    This coarse level representation allows to compute descent directions on 
    the fine level efficiently using a smaller subset of variables at the coarse level.
    }
    \label{fig:find-coarse-A}
\end{figure}

\subsection{Related Work}\label{sec:related-work}
\textit{Multilevel optimization} has been introduced by Nash \cite{Nash:2000}. 
Multilevel smooth unconstrained optimization is the setting of 
\cite{Nash:2000,Gratton:2008, WenGoldfarb:2009}, that includes trust-region  and
line search optimization.
The first multilevel optimization  method that covers the nonsmooth convex composite case was introduced in \cite{Parpas:2017}. It is a multilevel version of the well-known proximal gradient method. An accelerated variant was proposed in \cite{Parpas:2016_Magma}. Both approaches in \cite{Parpas:2017,Parpas:2016_Magma} can only handle smoothable functions in the sense of \cite{Beck:2012_smoothable_func}. A  multilevel optimization approach for box-constraints (non-smoothable functions) was proposed in \cite{Kocvara:2016}. The design of the coarse model in  \cite{Kocvara:2016} was inspired by the trust-region method in
\cite{Gratton:2008} and the work reported in \cite{Kornhuber93monotonemultigrid},
and was adopted also in \cite{Plier:2021}. 

Results for \textit{non-convex} optimization problems using a multilevel approach are not yet available in the Euclidean setting.
Recently, a geometric multilevel optimization approach has been proposed by \cite{Sutti:2020vn} for the specific case of low-rank matrix manifolds. Our approach differs in that we focus on information geometry for constructing multilevel problem representations,  which should be applicable to various constrained convex programs beyond the case study \eqref{eq:intro-approach} considered in this paper.

Since Nesterov's work \cite{Nesterov:2005aa}, \textit{accelerated} first-order convex optimization has become an active research field \cite{Tseng:2008aa,Salzo:2012aa,Wibisono_2016,Fazylab:2018aa}. Due to the performance of Nesterov's discrete scheme and the work \cite{Su:2016aa}, a major line of research concerns the understanding and generation of accelerated discrete scheme using continuous-time ODEs, where an adaptive `friction term' causes acceleration. In our approach, acceleration is solely achieved by employing a multilevel problem representation. Acceleration of the first-order Riemannian dynamics is an interesting problem of future research.

The use of \textit{geometry} for convex programming has a long history 
\cite{Bayer:1989aa,Nesterov:2002aa,Absil:2008aa}. The commonly adopted interior point geometry is based on Hessian metrics generated by self-concordant barrier functions, due to the provable optimality in connection with Newton-like second-order optimization \cite{Nesterov:1994aa}, see e.g. \cite{Parpas-self} for recent related work. Closer to our work is the comprehensive paper by \cite{Alvarez:2004} where Hessian metrics generated by convex functions of Legendre type are studied for a class of convex programs that include affine subspace constraints. 

In the present paper, we also consider such a Riemannian metric on a simpler structured bounded open convex feasible set, in order to focus on multilevel representation and accelerated first-order optimization that copes with large problem sizes. This necessitates, in particular, to devise restriction and prolongation operators not only for points on the manifold but also for tangent vectors. To this end, we employ information geometry in order to design smooth nonlinear mappings based on geometric averaging that can be efficiently computed in closed form. Numerical experiments demonstrate that our method outperforms the recent state-of-the-art method \cite{Hanzely:2021vc}.

\subsection{Contribution, Organization}\label{sec:contribution}

We introduce basic notation and briefly recall required properties of Bregman divergences in Section \ref{sec:preliminaries}. Section \ref{sec:two-level-Euclidean-optimization} summarizes the basic scheme of two-level optimization in Euclidean spaces. The core part of this paper, Section \ref{sec:Geometric-Optimization}, generalizes this scheme to a Riemannian setting. In particular, information geometry is employed in a way that indicates how other convex programs could be handled in the same way, in principle. While we indicate multilevel extensions in Sections \ref{sec:multilevel-Euclidean} and \ref{sec:multilevel-geometric}, respectively, we mainly focus on the core problem in this paper, that is two-level geometric optimization. Numerical experiments are reported in Section \ref{sec:Experiments} for a range of problem instances and compared to a recent state-of-the-art method \cite{Hanzely:2021vc}. We conclude in Section \ref{sec:conclusions}.


\section{Preliminaries}\label{sec:preliminaries}
\subsection{Notation}\label{sec:Notation}
We set $[n]=\{1,2,\dotsc,n\}$ for $n\in\N$ and
\begin{subequations}\label{eq:def-B-Bn}
\begin{align}
    \mc{B} &= (0,1), 
    \label{eq:def-B} \\ \label{eq:def-Bn}
    \mc{B}_{n} &= \mc{B}\times\dotsb\times\mc{B} \subset\R_{++}^{n}.\qquad (\text{open $n$-box})
\end{align}
\end{subequations}
$\eins = (1,1,\dotsc,1)^{\T}$ denotes the one-vector with dimension depending on the context. Sometimes we indicate the dimension by a subscript, e.g.~$\eins_{n}\in\R^{n}$. 
The componentwise application of functions like $\log$ is simply denoted as
\begin{equation}
    \log y = (\log y_{1},\dotsc,\log y_{n})^{\T},\quad y\in\R_{++}^{n}.
\end{equation}
Likewise, the componentwise multiplication (Hadamard product) and division of vectors $y, y'$ is written as
\begin{equation}
    y\cdot y' = (y_{1} y_{1}',\dotsc, y_{n}y_{n}')^{\T},\qquad
    \frac{y}{y'} = \Big(\frac{y_{1}}{y_{1}'},\dotsc,\frac{y_{n}}{y_{n}'}\Big)^{\T}.
\end{equation}
The Euclidean vector inner product is denoted by $\la y, y'\ra = y^{\T} y'$. 
$\partial f$ denotes the gradient of a differentiable function $f\colon \R^{n}\to\R$, whereas $\ggrad f$ denotes the Riemannian gradient of a function $f\colon\mc{M}\to\R$ on a manifold $(\mc{M},g)$ with metric $g$.

Components $y_{j}$ of vectors $y\in\mc{B}_{n}$ are indexed by $[n]$ that we specifically denote in this context by
\begin{equation}
    I_{n}=[n].\qquad (\text{$n$-grid})
\end{equation}
The notion ``$n$-grid'' reminds of the fact that indices $j$ typically refer to locations in an underlying physical domain (cf.~Section \ref{sec:Introduction}). In this sense, we will synonymously refer to $j$ as `point' and `index'.
A disjoint partition
\begin{equation}
    I_{n} = I_{m}\dot\cup I_{m}^{c},\quad m<n
\end{equation}
defines a coarse grid $I_{m}$ indexing coarse grid vectors $x\in\mc{B}_{m}$ and fine grid points indexed by the complement $I_{m}^{c}$. $I_{n}$ is called fine grid and vectors $y\in\mc{B}_{n}$ are called fine grid vectors.

Regarding problem \eqref{eq:intro-approach}, we summarize our assumptions. 
\begin{equation}\tag{A}
\begin{aligned}
b_{i} &> 0,\qquad i=1,\dotsc,p 
&\qquad&\text{(1)} \\
A_{ij} &\geq 0,\qquad i=1,\dotsc,p,\; j=1,\dotsc,n 
&&\text{(2)} \\
\forall y > 0\colon\quad \supp(A y) &= p,
&&\text{(3)} \\
\rank A &= \min\{p,n\},
&&\text{(4)}
\end{aligned}
\end{equation}
where $\supp(y)$ denotes the number of nonzero components of a vector $y$.

\subsection{Bregman Divergences}

\begin{definition}[Bregman divergence {\cite{Bregman:1967tn}, \cite[Section 2]{Censor:1981vy}, \cite[Def.~7.6.1]{Borwein:2010aa}}]\label{def:Bregman}
Let $E$ be a Euclidean space and $\phi \colon E\to (-\infty,+\infty]$ be a convex function of Legendre type with non-empty domain $\dom \phi \subseteq E$ \cite[Def.~7.1.1]{Borwein:2010aa}. The Bregman divergence corresponding to $\phi$ is defined by
\begin{equation}\label{eq:def-Df}
    D_{\phi} \colon E\times \intr\dom \phi \to[0,\infty]\colon \qquad
    (x,y)\mapsto \phi(x)-\phi(y)-\la \partial \phi(y),x-y\ra.
\end{equation}
\end{definition}
Bregman divergences have the properties $D_{\phi}(x,y)\geq 0$ and $D_{\phi}(x,y)=0$ if and only if $x=y$. 
A basic example is the Kullback-Leibler (KL) divergence
\begin{subequations}\label{eq:def-KL}
\begin{align}
    D_{\vphi}(y,y') &= \Big\la\eins,y \cdot \log\frac{y}{y'} + y'-y\Big\ra
\intertext{corresponding to the function}
\vphi(y) &= \la y,\log y\ra - \la\eins,y\ra,\quad y\in\R_{+}^{n}.
\end{align}
\end{subequations}
\begin{lemma}[{\cite[Lem 3.1]{Chen_Teboulle_1993}}]\label{Pythagoras}
Let $S \subset \R^n$ be an open set with closure $\ol{S}$, and let $\phi \colon\ol{S}\to \R$
be a convex function as in Definition \ref{def:Bregman}. Then, for any
three points $a, b \in S$ and $c\in\ol{S}$, the identity
\begin{equation}
D_{\phi}(c, a) + D_{\phi}(a, b) - D_{\phi}(c, b)
= \la \partial \phi(b) -  \partial \phi(a), c - a \ra
\end{equation}
holds.
\end{lemma}
Now we consider the data term of \eqref{eq:intro-approach}
\begin{equation}\label{eq:hBreg}
h(y):=D_\varphi(Ay,b)
\end{equation}
with the KL-divergence $D_{\vphi}$ given by \eqref{eq:def-KL} and gradient
\begin{equation}\label{eq:nabla_hBreg}
\partial h(y)= A^\top \log\frac{Ay}{b}.
\end{equation}
\begin{lemma}\label{Bregmangap_data}
Assume $b\in \R_{++}^{p}$, $y\in\R_{+}^{n}$,  $y'\in\R_{++}^{n}$ and $A y'\in \R_{++}^{p}$. Then the Bregman divergence corresponding to the function $h$ given by \eqref{eq:hBreg} reads
\begin{equation} \label{Breg_gap_withoutb}
D_h(y,y') = D_\varphi(Ay, Ay').   
\end{equation}
\end{lemma}
\begin{proof}
By definition \eqref{eq:def-Df}, we have
\begin{align*}
D_h(y,y') & = h(y) -h(y') - \la \partial h(y'), y-y' \ra\\
 & \overset{\eqref{eq:hBreg}}{=} D_\varphi(Ay,b)  - D_\varphi(Ay',b) -  \la \partial\varphi (Ay') -\partial\varphi(b), Ay - Ay' \ra. \numberthis  \label{eq:proof-Dh} 
\end{align*}
Substituting $\phi \leftarrow\vphi$, $b\leftarrow b$, $a\leftarrow Ay'$ and $c\leftarrow Ay$ in Lemma \ref{Pythagoras} yields
\begin{equation*}
    D_\varphi(Ay,Ay') + D_\varphi(Ay',b) - D_\varphi(Ay,b)
    = \la \partial\varphi(b) - \partial\varphi (Ay') , Ay - Ay' \ra.
\end{equation*}
Comparing this equation and \eqref{eq:proof-Dh} is equivalent to \eqref{Breg_gap_withoutb}.
\end{proof}


\section{Two-Grid Euclidean Optimization}\label{sec:two-level-Euclidean-optimization}

In this section, we summarize prior work on two-level optimization in Euclidean spaces \cite{Nash:2000,WenGoldfarb:2009}. Our main contribution in this paper, the generalization to a Riemannian setting, is described in Section \ref{sec:Geometric-Optimization}.

\subsection{Unconstrained Two-Grid Euclidean Optimization}
As motivated in Section \ref{sec:Introduction}, we assume a generic objective function $f$ to be given that can be evaluated at different discretization levels. In this section, we consider two discretization levels called fine grid and coarse grid, respectively, and use the following notation. 
\begin{subequations}\label{eq:notation-Euclidean-twolevel}
\begin{align}
&\text{fine grid variable}\colon y \in\R^n  , 
&
&\text{current iterate}\colon y_0\in\R^n, \\
&\text{coarse grid variable}\colon x\in \R^m , 
&
&\text{restriction of $y_0$} \colon x_0 = Ry_0\in \R^m.
\end{align}
\end{subequations}

Here, the restriction map $R$ and its counterpart, the prolongation map $P$,
\begin{subequations}\label{eq:R-P-Euclidean}
\begin{align}
    R&\colon \R^{m}\to\R^{n}, \\
    P&\colon \R^{n}\to\R^{m},
\end{align}
\end{subequations}
transfer fine and coarse grid variables to the coarse and fine grid, respectively, typically via linear interpolation or simple injection as in classical multigrid methods \cite{Trottenberg:2001}.

We assume that the \textit{argument determines} if the objective function $f$ is evaluated on the fine grid or the coarse grid, respectively. Thus, in view of \eqref{eq:notation-Euclidean-twolevel}, $f(y)$ implies $f\colon\R^{n}\to\R$ whereas $f(x)$ implies $f\colon\R^{m}\to\R$, with $m<n$.

We next describe a two grid cycle that computes an update $y_1$ of $y_{0}$ at a fine grid. This is done either along a search direction efficiently obtained from a surrogate objective function  defined on the coarse grid using a much smaller number of variables: \emph{coarse correction}. Alternatively, whenever a
coarse correction is not effective, the update is determined using a standard local objective function approximation defined on the fine grid: \emph{fine correction}.

The general two grid approach is summarized as Algorithm \ref{alg:twogrid-Euclidean}. It uses a representation $\psi$ of the objective function, specified in Section \ref{sec:coarse-model-Euclidean}, that is \textit{not} equal to the evaluation of the objective $f\colon\R^{m}\to\R$ on the coarse grid.
\begin{algorithm}[ht]
\DontPrintSemicolon
\LinesNumbered
\SetKwInOut{Input}{input}
\SetKwInOut{Output}{output}
\Input{initial point $y_{0}$, mappings $R, P$ (cf.~\eqref{eq:R-P-Euclidean}, coarse grid model $\psi$ (cf.~\eqref{eq:def-psi-Euclidean}).}
\Output{approximate global minimizer $y^{\ast}$.}
Initialization: $k=0$ \;
\Repeat{a stopping rule is met.}{
  \If{condition to invoke the coarse model is satisfied at $y_0$ \label{if-cond-coarse}}{
    set $x_0=R y_0$ \label{MLO_general_line4} \;
    determine a coarse grid iterate $x$ such that ${\psi(x)} < f(x_0)$ \label{MLO_general_line6} \;
    set $d=P (x-x_0)$ \label{MLO_general_line7} \;
    find  $\alpha>0$ such that $f(y_0+\alpha d) < f(y_0)$ \label{MLO_general_line8} \tcc*{line search} 
    update $y_0 \leftarrow y_0 + \alpha d$ \label{MLO_general_line9}\; 
  }\Else{
  apply one iteration of a monotone fine level algorithm to find $y_1$ with $f(y_1)< f(y_0)$ and update
  $y_0 \leftarrow y_1$  \label{MLO_general_line10}\;
  }
  increment $k \leftarrow k+1$ \;
  }
\caption{Two Level Euclidean Optimization}\label{alg:twogrid-Euclidean}
\end{algorithm}

The condition in line \ref{if-cond-coarse} is specified in Section \ref{sec:invoke-coarse-model}.
The role of the coarse grid model $\psi$ is further discussed in the following section.

\subsubsection{Coarse Grid Model}\label{sec:coarse-model-Euclidean}
The core ingredient of two level optimization is a \emph{coarse grid model}, that is a coarse grid  representation of the fine grid problem in terms of the objective function $f$ evaluated at the coarse grid, the current iterate $y_0$ and its restriction $x_{0}=Ry_{0}$ to the coarse grid. It reads
\begin{subequations}\label{eq:def-psi-Euclidean}
\begin{align}
    \psi(x)
    &=f(x)- \la x-x_{0},\kappa(x_{0},y_{0})\ra,  \label{eq:def-psi-Euclidean-b}\\
    \intertext{with the linear correction term }
    \kappa(x_{0},y_{0})  & = \partial f(x_{0})- R \partial f({y}_0).
\end{align}
\end{subequations}
The rationale behind the coarse grid model is to determine efficiently a gradient-like descent direction on the \textit{fine} grid using a much smaller number of variables on the coarse grid, as first proposed in \cite{Nash:2000}.
The linear correction term \eqref{eq:def-psi-Euclidean-b} guarantees that a fine grid stationary point is also a coarse grid stationary point. More precisely, the gradient of $\psi$ equals
\begin{subequations}
\begin{align}
   \partial \psi(x) &= \partial f(x) - \kappa(x_{0},y_{0}) \\
   &=  \partial f(x) - 
   \partial f(x_{0})+ R \partial f({y}_0) \\
   &= \partial \left(f(x) - 
   f(x_{0})\right)+ R \partial f({y}_0),
\end{align}
\end{subequations}
and hence the \emph{first-order coherence condition} \cite{Nash:2000}
\begin{equation}\label{eq:first-order-coherence}
   \partial \psi(x_0)  =  R \partial f({y}_0),
\end{equation}
holds at $x=x_0$.

\subsubsection{Descent Directions Based on the Coarse Grid Model}\label{sec:Coarse-descent-direction}
We next show that a sufficient decrease of $\psi$ given by \eqref{eq:def-psi-Euclidean}
by updating the smaller number of coarse grid variables $x$ determines a descent direction for the fine grid objective. To this end, we
rewrite  $\psi$ in terms of the Bregman divergence corresponding to the coarse grid objective $f$,
\begin{subequations}
\begin{align}
    \psi(x)
    &=f(x)- \la x-x_{0},\mc{\kappa}(x_0,y_0)\ra
    \\
    &=f(x)- \la x-x_{0},\partial f(x_{0})- R \partial f({y}_0)\ra
    \\
    &=f(x)- f(x_0) - \la x-x_{0},\partial f(x_{0})- R \partial f({y}_0)\ra + f(x_0)\\
    &=D_f(x,x_0) + \la x-x_{0}, R \partial f({y}_0)\ra + f(x_0).
\end{align}
\end{subequations}
Now, assuming $R$ and $P$ satisfy the Galerkin condition 
\begin{equation}\label{eq:Galerkin}
    R=P^\T
\end{equation}
that is a standard condition in multigrid literature \cite{Trottenberg:2001}, we get
\begin{equation}\label{def:psi-Bregman}
\psi(x)
=D_f(x,x_0) + \la P(x-x_{0}), \partial f({y}_0)\ra + f(x_0).
\end{equation}
Hence, whenever
\begin{equation}\label{eq:psi_decrease}
\psi(x) < f(x_0) = \psi(x_0)
\end{equation}
at a coarse iterate $x$ and the coarse objective $f$ is convex, then the
vector 
\begin{equation}\label{eq:coarse_correction}
d =P(x-x_{0})
\end{equation}
is a descent direction of the fine grid objective $f$, since \eqref{eq:psi_decrease} and $D_{f}\geq 0$ implies 
\begin{equation}
   \la P(x-x_{0}), \partial f({y}_0)\ra <0
\end{equation}
by \eqref{def:psi-Bregman}.
To find such a coarse grid iterate $x$ in practice, one typically employs few iterations of a monotone algorithm for minimizing $\psi$ at the coarse level.

\begin{remark}\label{rem:psi-Df} The coarse grid model rewritten in the form \eqref{def:psi-Bregman} reveals that it
incorporates both 
second-order information of the coarse grid objective  (first term) and  first-order information of the fine grid objective (second term). 
Indeed, ignoring the constant $f(x_{0})$ and assuming $f$ is twice differentiable, the first term can be equivalently rewritten as

\begin{equation}
D_{{f}}(x , x_0) =
 \frac{1}{2}\left \langle x-x_0,\partial^{2}{f}(\wt x) ( x-x_0)\right \rangle 
\end{equation}

for some $\wt x \in \{(1-t) x+t x_0 \colon t\in[0,1]\}$, where $\partial^{2}f(\wt x)$ denotes the Hessian of $f$ at $\wt x$.
\end{remark}

\subsubsection{Coarse Grid Correction Criteria}\label{sec:invoke-coarse-model}
We employ the coarse grid model whenever the criteria
\begin{equation}\label{coarse_condition}
\|R \partial f(y_0)\| \ge \eta \|\partial f(y_0)\|\quad \text{and}\quad \|y_0 - y_c\|> \varepsilon
\end{equation}
are satisfied, where $\eta \in(0,\min(1,\|R\|))$, $\varepsilon \in (0,1)$ and $y_c$ is the last point that initiated a coarse grid correction.

\begin{enumerate}[(1)]
    \item The first condition in \eqref{coarse_condition} is adopted from \cite{WenGoldfarb:2009} and is widely used in the multilevel optimization literature.
    We discuss this first condition in view of the first-order coherence condition \eqref{eq:first-order-coherence}.
    
    If the left-hand side $\|R\partial f(y_{0})\|$ is too small at the current iterate $y_{0}$, then there is nothing to optimize on the coarse grid -- as $\|\partial \psi(x_{0})\|$ is small too. The entire first condition says that $y_{0}$ is non-optimal in a twofold sense: $\|\partial f(y_{0})\|>0$ (otherwise both sides vanish) and on a coarser scale $f$ is sufficiently non-flat as well such that invoking the coarse grid model pays off. 
    \item Violation of the second condition says that the current iterate $y_{0}$ is too close to a point that already initiated a coarse grid correction.
\end{enumerate}

\subsubsection{Application to a Regularized Inverse Problem} 

In order to apply two-level optimization to the regularized inverse problem \eqref{eq:intro-approach}, the coarse grid model function $\psi$ given by \eqref{def:psi-Bregman} has to be computed. Similar to the evaluation of the objective function at both levels, we assume that the operator $A$ can be directly evaluated at both levels. For example, in discrete tomography, the projection matrix $A$ represents the incidence relation of projection rays and cells centered at the grid points, and can be evaluated on every grid. 
 
If suffices to consider the Bregman divergence since the remaining term only involves fine grid variables (we ignore the constant $f(x_{0})$). 
 
 \begin{lemma} \label{lem:data_term_drops} The Bregman divergence corresponding to the objective function \eqref{eq:intro-approach}, evaluated at the coarse level, reads
 \begin{align}
 D_{f}( x , x_0) = D_\varphi(Ax  , Ax_0) + \lambda D_{J}( x , x_0).
 \end{align}
\end{lemma}
\begin{proof} Lemma \ref{Bregmangap_data} shows $D_h( x , x_0) = D_\varphi(A x  , A x_0)$ with $h$ given by \eqref{eq:hBreg}. The result follows from linearity of the Bregman divergence with respect to the inducing function, i.e.~$D_{h+\lambda J}(x,x_0)=D_{h}(x,x_0)+\lambda D_{J}(x,x_0)$.
\end{proof}

\begin{remark} 
Lemma \ref{lem:data_term_drops} 
shows, in particular, that the specific ray geometry (depicted in Fig. \ref{fig:find-coarse-A}) at the coarse level can be selected independently of the ray geometry at the fine level, since the projection data $b$ drop out and hence have not to be transferred between levels.
\end{remark}


\subsection{Multilevel Extension}\label{sec:multilevel-Euclidean}

We introduce the notation
\begin{subequations}
\begin{align}
    x^{(\ell)} &\in \R^{n_{\ell}}
    \intertext{for the variables at levels $\ell\in\{0,1,2,\dotsc\}$, where} 
    x^{(0)} := x &\in \R^{n} =: \R^{n_{0}}
\end{align}
are the original fine grid variables. The components $x^{(\ell)}_{i}$ of these vectors are indexed by the nested index sets
\begin{equation}
    i\in I_{n_{\ell}} \subset I_{n_{\ell-1}} \subset \dotsb \subset I_{n} = I_{n_{0}}.
\end{equation}
\end{subequations}
Accordingly, the coarse grid model \eqref{eq:def-psi-Euclidean} applied at level $\ell$ reads
\begin{subequations}
\begin{align}
    \psi(x^{(\ell)}) &= f(x^{(\ell)})-\la x^{(\ell)}-x^{(\ell)}_{0},\kappa(x_{0}^{(\ell)},x_{0}^{(0)})\ra,
    \\
    \kappa(x_{0}^{(\ell)},x_{0}^{(0)}) &= \partial f(x_{0}^{(\ell)}) - R^{(0)}_{(\ell)}\partial f(x_{0}^{(0)}),
\end{align}
\end{subequations}
with $x_{0}^{(\ell)} = R^{(0)}_{(\ell)} x_{0}^{(0)}$.
Substituting $\kappa(x_{0}^{(\ell)},x_{0}^{(0)})$ into $\psi(x^{(\ell)})$, we have
\begin{align*}
    \psi(x^{(\ell)})-f(x^{(\ell)}_{0}) 
     = f(x^{(\ell)})-f(x^{(\ell)}_{0})-\big\la x^{(\ell)}-x^{(\ell)}_{0},\partial f(x_{0}^{(\ell)}) - R^{(0)}_{(\ell)}\partial f(x_{0}^{(0)})\big\ra
    \\ \label{eq:psi-ell}
    = D_f(x^{(\ell)},x^{(\ell)}_{0}) +\big\la P_{(0)}^{(\ell)}(x^{(\ell)}-x^{(\ell)}_{0}), \partial f(x_{0}^{(0)})\big\ra 
    \numberthis 
    \\
    \overset{\text{Lem.}~\ref{Bregmangap_data}}{=} D_\varphi(A^{(\ell)}x^{(\ell)},A^{(\ell)}x^{(\ell)}_{0}) + \lambda D_{J}(x^{(\ell)},x^{(\ell)}_{0}) +
    \big\la P_{(0)}^{(\ell)}(x^{(\ell)}-x^{(\ell)}_{0}), \partial f(x_{0}^{(0)})\big\ra,
\end{align*}
where $A^{(\ell)}\in R^{{p_{\ell}\times {n_{\ell}}}}$ is the discretization of the linear operator
at level $\ell$, $R_{(\ell)}^{(0)} = (P_{(0)}^{(\ell)})^{\T}$ and $P_{(0)}^{(\ell)}$ prolongates vectors at level $\ell$ to level $0$ by interpolation, e.g.~as through the composition $P_{(0)}^{(\ell)}=P_{(\ell-1)}^{(\ell)}\circ \dotsb\circ P_{(0)}^{(1)}$ of interpolation operators between adjacent levels.


We note that the data vector $b\in \R^{p_0}:=\R^{p}$
in \eqref{eq:intro-approach} is only required on the finest (original) level. In addition, comparing \eqref{def:psi-Bregman} and \eqref{eq:psi-ell} makes evident that the discussion of \eqref{def:psi-Bregman} applies also to \eqref{eq:psi-ell}: an iterate $x^{(\ell)}$ satisfying $\psi(x^{(\ell)})-f(x^{(\ell)}_{0})<0$ yields a descent direction $P_{(0)}^{(\ell)}(x^{(\ell)}-x^{(\ell)}_{0})$ at level $0$. Clearly, the criterion analogous to \eqref{coarse_condition} (with $R_{(\ell)}^{(0)}$ in place of $R$) 
for employing a coarse correction  will be less likely satisfied
at level $\ell$ than a finer level $\ell'<\ell$.
On the other hand, if it is satisfied, then determining a descent direction is computationally (very) cheap. 

Optimizing the costs for invoking coarse grid models vs.~saving computations through evaluating these models with much smaller numbers of variables, is left for future work.


\section{Multilevel Geometric Optimization}\label{sec:Geometric-Optimization}

\subsection{Riemannian Geometry of the Box}\label{sec:geometry-box}
In this section, we represent the open $n$-box $\mc{B}_{n}$ \eqref{eq:def-Bn} as a Riemannian manifold. To this end, we turn the open intervall \eqref{eq:def-B} into a manifold $(\mc{B},g)$ with metric $g$ and define $(\mc{B}_{n},g)$ as the corresponding product manifold.

In order to specify $(\mc{B},g)$, we apply basic information geometry \cite{Amari:2000aa}. Let points $\eta\in\mc{B}$ parametrize the Bernoulli distribution
\begin{equation}\label{eq:p-Bernoulli}
    p(z;\eta) = \eta^{z}(1-\eta)^{1-z}
\end{equation}
of a binary random variable $Z\in\{0,1\}$. Then the metric tensor of the Fisher-Rao metric is simply a scalar function of $\eta$ given by
\begin{equation}\label{eq:G-B}
    G(\eta)
    = 4\sum_{z\in\{0,1\}}\Big(\frac{d}{d\eta}\sqrt{p(z;\eta)}\Big)^{2}
    = \frac{1}{1-\eta} + \frac{1}{\eta}
    = \frac{1}{\eta(1-\eta)}.
\end{equation}
In view of \eqref{eq:def-Bn}, this naturally extends to the metric $g$ on $T\mc{B}_{n}$ in terms of the diagonal matrix
\begin{equation}
    G_{n}(y) = \Diag\Big(\frac{1}{y_{1}(1-y_{1})},\dotsc,\frac{1}{y_{n}(1-y_{n})}\Big)
\end{equation}
as metric tensor, using again the symbol `$G$' for simplicity. We naturally identify \begin{equation}\label{eq:ident-TB}
    T_{y}\mc{B}_{n}\cong \R^{n},\quad\forall y\in\mc{B}_{n}
\end{equation}
and denote this metric interchangeably by
\begin{equation}\label{eq:metric-notation}
    g_{y}(v,v') = \la v,v'\ra_{y} = \la v,G_{n}(y) v'\ra,\qquad
    \forall v,v'\in T_{y}\mc{B}_{n}.
\end{equation}

\subsection{Retraction and its Inverse}

Retractions \cite[Def.~4.1.1]{Absil:2008aa} and their inverses are basic ingredients of first-order optimization algorithms on Riemannian manifolds. The main motivation is to replace the exponential map with respect to the metric (Levi Civita) connection by an approximation that can be efficiently evaluated or even in closed form. Below, we compute the exponential map with respect to the e-connection of information geometry \cite{Amari:2000aa} and show subsequently that it is a retraction.
\begin{proposition}\label{prop:exp}
The exponential maps on $\mc{B}$ resp. $\mc{B}_{n}$ with respect to the e-connection are given by
\begin{subequations}\label{eq:def-exp}
\begin{align}
    \exp&\colon\mc{B}\times T\mc{B}\to\mc{B},
    &
    \exp_{\eta}(t v) 
    &= \frac{\eta e^{t \frac{v}{\eta (1-\eta)}}}{1-\eta + \eta e^{t \frac{v}{\eta (1-\eta)}}},\quad t > 0
    \label{eq:exp-B} \\ \label{eq:exp-Bn}
    \exp&\colon\mc{B}_{n}\times T\mc{B}_{n}\to\mc{B}_{n},
    &
    \exp_{y}(t v) 
    &= \big(\exp_{y_{j}}(t v_{j})\big)_{j\in I_{n}}
\end{align}
\end{subequations}
with inverses
\begin{subequations}\label{eq:def-inv-exp}
\begin{align}
    \exp^{-1}&\colon\mc{B}\times\mc{B}\to T\mc{B},
    &
    \exp_{\eta}^{-1}(\eta')
    &= \eta (1-\eta)\log\frac{(1-\eta)\eta'}{\eta (1-\eta')}
    \label{eq:inv-exp-B} \\ \label{eq:inv-exp-Bn}
    \exp^{-1}&\colon\mc{B}_{n}\times\mc{B}_{n} \to  T\mc{B}_{n},
    &
    \exp_{y}^{-1}(y')
    &= \big(\exp_{y_{j}}^{-1}(y_{j}')\big)_{j\in I_{n}}.
\end{align}
\end{subequations}
\end{proposition}
\begin{proof}
It suffices to show \eqref{eq:exp-B} from which \eqref{eq:inv-exp-B} follows from an elementary calculation. \eqref{eq:exp-Bn} and \eqref{eq:inv-exp-Bn} result from the componentwise application of these maps.

A key concept of information geometry is to replace the metric connection by a pair of connections that are dual to each other with respect to the Riemannian metric $g$ \cite[Section 3.1]{Amari:2000aa}. In particular, under suitable assumptions, the parameter space of a probability distribution becomes a Riemannian manifold that is dually flat, i.e.~two distinguished coordinate systems (the so-called m- and e-coordinates) exists with affine geodesics. We consider the simple case $(\mc{B},g)$.

First, we rewrite the distribution \eqref{eq:p-Bernoulli} as distribution of the exponential family \cite{Brown:1986vy}
\begin{subequations}\label{eq:exp-parametrization}
\begin{align}
    p(z;\theta) &= \exp\big(z \theta-\psi(\theta)\big)
    \\ \label{eq:theta-eta}
    \theta &= \theta(\eta) = \log\frac{\eta}{1-\eta},\qquad
    \eta = \eta(\theta) = \frac{e^{\theta}}{1+e^{\theta}}
\end{align}
\end{subequations}
with exponential parameter $\theta$ and log-partition function 
\begin{equation}\label{eq:def-psi}
    \psi(\theta)=\log Z(\theta),\quad
    Z(\theta)=1+e^{\theta}
\end{equation}
that is convex and of Legendre type \cite[Def.~7.1.1]{Borwein:2010aa}. The aforementioned distinguished two coordinates are $\eta$ and $\theta$ with affine geodesics
\begin{subequations}\label{eq:affine-geodesics}
\begin{align}
    t &\mapsto \eta_{v}(t) = \eta + t v \in \mc{B}, 
    \label{eq:m-affine} \\ \label{eq:e-affine}
    t &\mapsto \theta_{u}(t) = \theta + t u \in \R.
\end{align}
\end{subequations}
Note that unlike $\eta, v$, the coordinate $\theta$ and the tangent $u$ are unconstrained. Using \eqref{eq:theta-eta}, the e-geodesic reads
\begin{equation}\label{eq:proof-eta-theta-u}
    \eta\big(\theta_{u}(t)\big) = \frac{e^{\theta_{u}(t)}}{1+e^{\theta_{u}(t)}}
    = \frac{e^{\theta} e^{t u}}{1 + e^{\theta} e^{t u}} \in\mc{B}.
\end{equation}
It remains to replace $\theta$ by $\eta = \eta(\theta)$ due to \eqref{eq:theta-eta} and $u$ by $\eta$ and $v$, respectively. The latter results from substituting the affine geodesic \eqref{eq:m-affine} into $\theta = \theta(\eta)$ and to express $u$ by
\begin{subequations}
\begin{align}
    u = u(\eta,v) &
    = \frac{d}{dt}\theta\big(\eta_{v}(t)\big)\big|_{t=0}
    = \frac{d}{dt}\theta(\eta + t v)\big|_{t=0} \\
    &= \frac{d}{dt}\log\frac{\eta+t v}{1-\eta-t v}\Big|_{t=0}
    = \frac{v}{\eta (1-\eta)}.
\end{align}
\end{subequations}
Substituting into \eqref{eq:proof-eta-theta-u} yields
\begin{equation}
    \exp_{\eta}(t v) 
    := \frac{e^{\theta(\eta)} e^{t u(\eta,v)}}{1 + e^{\theta(\eta)} e^{t u(\eta,v)}}
    = \frac{\eta e^{t \frac{v}{\eta (1-\eta)}}}{1-\eta + \eta e^{t \frac{v}{\eta (1-\eta)}}}
\end{equation}
which is \eqref{eq:exp-B}.
\end{proof}
\begin{remark}[$g$ is a Hessian metric]
The dual nature of the exponential parametrization \eqref{eq:exp-parametrization} is also highlighted by recovering the metric tensor \eqref{eq:G-B} as Hessian metric from the potential $\phi(\eta)$ that is conjugate to the log-partition function \eqref{eq:def-psi},
\begin{equation}\label{eq:def-phi-psi-ast}
    \phi(\eta) = \psi^{\ast}(\eta) 
    = \eta\log\eta + (1-\eta)\log(1-\eta),
\end{equation}
that is
\begin{equation}
        \phi''(\eta) = G(\eta) = \frac{1}{\eta (1-\eta)}.
\end{equation}
The dual coordinate and potential yield the inverse metric tensor
\begin{equation}
    \psi''(\theta) = \frac{e^{\theta}}{(1+e^{\theta})^{2}}
    = \frac{1}{G(\eta)}\bigg|_{\eta=\eta(\theta)}.
\end{equation}
\end{remark}
The following formulas will be used later on.
\begin{lemma}[Differential of $\exp$ and $\exp^{-1}$]\label{eq:dexp-dexp-1}
The differentials of the mappings $\exp_{\eta}$ and $\exp_{\eta}^{-1}$ at $u$ and $\eta'$, respectively, are given by
\begin{subequations}\label{eq:depx-dinvexp}
\begin{align}
    d\exp_{\eta}(u)\colon T_{\eta}\mc{B}&\to T_{\eta'}\mc{B}, 
    \\ \label{eq:dexp} 
    v \mapsto d\exp_{\eta}(u) v
    &= \frac{e^{\frac{u}{\eta (1-\eta)}}}{\big(1-\eta + \eta e^{\frac{u}{\eta (1-\eta)}}\big)^{2}} v
    = \frac{\eta' (1-\eta')}{\eta (1-\eta)} v,\qquad \eta' = \exp_{\eta}(u)
    \\
    d\exp_{\eta}^{-1}(\eta')\colon T_{\eta'}\mc{B}&\to T_{\eta}\mc{B}, 
    \\ \label{eq:dinvexp}
    v' \mapsto d\exp_{\eta}^{-1}(\eta') v'
    &= \frac{\eta (1-\eta)}{\eta' (1-\eta')} v'.
\end{align}
\end{subequations}
\end{lemma}
\begin{proof}
By direction computation,
\begin{equation}
    d\exp_{\eta}(u) v = \frac{d}{dt} \exp_{\eta}(u+t v)\big|_{t=0}
\end{equation}
using \eqref{eq:exp-B} and similarly for \eqref{eq:dinvexp} using \eqref{eq:inv-exp-B}. Clearly, $d(\exp_{\eta}^{-1}(\eta')) = (d\exp_{\eta}(u))^{-1}$.
\end{proof}
Retractions provide a proper class of surrogate mappings for replacing the canonical exponential map corresponding to the metric connection.
\begin{proposition}[$\exp$ is a retraction]\label{prop:retraction}
The mapping $\exp\colon T\mc{B}\to \mc{B}$ is a retraction in the sense of \cite[Def.~4.1.1.]{Absil:2008aa}.
\end{proposition}
\begin{proof}
We check the two criteria that characterize retractions. First,
\begin{equation}
    \exp_{\eta}(0) \overset{\eqref{eq:exp-B}}{=} \eta,\qquad 
    \forall \eta\in\mc{B}.
\end{equation}
Second, the so-called local rigidity condition
\begin{equation}
    d\exp_{\eta}(0) \overset{\eqref{eq:dexp}}{=} 1 = \mrm{id}_{T_{\eta}\mc{B}},\qquad 
    \forall \eta\in\mc{B}
\end{equation}
holds as well.
\end{proof}

\begin{remark}[Vector transport]\label{rem:vector-transport}
In the same way as retractions may properly replace the canonical exponential map, the class of vector transport mappings 
\begin{equation}
    \Pi\colon T\mc{B}\times T\mc{B}\to T\mc{B},\qquad
    (u_{\eta},v_{\eta}) \mapsto \Pi_{u_{\eta}}(v_{\eta})
\end{equation}
may properly replace the canonical parallel transport of tangent vectors \cite[Section 8.1]{Absil:2008aa}. Given a retraction, like in Proposition \ref{prop:retraction}, a vector transport map can be conveniently obtained by differentiation \cite[Section 8.1.2]{Absil:2008aa},
\begin{equation}\label{eq:def-Piuv}
    \Pi_{u}(v) = d\exp_{\eta}(u)v
    \overset{\eqref{eq:dexp}}{=} \frac{\eta' (1-\eta')}{\eta (1-\eta)} v,\qquad
    \eta' = \exp_{\eta}(u).
\end{equation}
The above proof of Proposition \ref{prop:exp} revealed the representation of tangents $v$ and $u$ of \eqref{eq:affine-geodesics} in the form
\begin{subequations}
\begin{align}
     T_{\eta}^{(m)} &= \{v\in\R\colon \eta + v \in\mc{B}\},\quad \eta\in\mc{B},
     \label{eq:m-tangent} \\ \label{eq:e-tangent}
     T^{(e)}_{\eta} &= \Big\{u = \frac{v_{u}}{\eta (1-\eta)}\colon v_{u}\in\R\Big\},\quad \eta\in\mc{B}.
\end{align}
\end{subequations}
Thus, in view of \eqref{eq:e-tangent}, the vector transport \eqref{eq:def-Piuv} may be interpreted as sending $u = \frac{v}{\eta (1-\eta)} \in T_{\eta}^{(e)}$ to $u' = \frac{v'}{\eta' (1-\eta')} \in T_{\eta'}^{(e)}$ with $v'=\Pi_{u}(v)$.
\end{remark}

\subsection{Line Search}
We now adopt the simplest approach to optimizing a differentiable function
$f\colon \mc{B}_{n}\to \R$
on the manifold $(\mc{B}_{n},g)$,
by continuously translating a current iterate $y$ in the direction of steepest descent, 
$-\ggrad f(y)$. 
With the retraction calculated in Proposition~\ref{prop:exp}, we will now use Armijo~line~search~\cite[Def. 4.2.2]{Absil:2008aa} to move within our Riemannian manifold.
For minimizing coarse level approximations, see Section \ref{sec:two-level-geometric}, we will employ the same techniques.

In particular, given $f$, a point $y \in\mc{B}_{n}$ and  $\eta :=-\ggrad f(y)\in T\mc{B}_{n}$, the goal is to determine an update of $y$ along the steepest descent direction $\eta$. To this end, we set
\begin{equation}
y_{+} = f\big(\exp_{y}(\alpha \eta)\big),\quad \alpha \in\R_+,
\end{equation}
with $\exp_{y}(\cdot)$ given by \eqref{eq:exp-B}. 
The appropriate value for the step size parameter $\alpha>0$ determines the update $\exp_{y}(\alpha \eta)$ of $y$. An exact line search, which would find $\alpha = \argmin_{t>0} f(\exp_{y}(t \eta))$, is often not feasible. A commonly used alternative is the Armijo~line~search, which starts with a large estimate for $\alpha$ and gradually reduces it until a value is found that results in a sufficient decrease of the objective function $f$. Algorithm \ref{alg:RG-Armijo} details the the Armijo~line~search strategy, and Theorem \ref{thm:ALS-convergence} guarantees convergence of the overall procedure.
\begin{theorem} \cite[Thm. 4.3.1]{Absil:2008aa}\label{thm:ALS-convergence}
Any cluster point of the sequence generated by the Algorithm \eqref{alg:RG-Armijo} is a critical point of $f$.
\end{theorem}






\begin{algorithm}[H]
\DontPrintSemicolon
\textbf{initialization:} initial  point $y$, initial step size $\alpha_0>0$, $\sigma, \beta \in (0,1)$.\;
 set search direction to $\eta = - \ggrad f(y)$; \;
\Repeat{termination criterion holds.}{
  set $\alpha =\alpha_0$; \;
  \If{$f(\exp_{y}(\alpha \eta)) - f(y) > \sigma \alpha \langle \ggrad f(y), \eta \rangle_{y}$}{
  set $\alpha=\beta \alpha$; \;
  }
  update the iterate $y = \exp_{y}(\alpha \eta)$; \;
  update the search direction to $\eta = - \ggrad f(y)$; \;
  }
\caption{Riemannian  Gradient with Armijo Line Search \cite{Absil:2008aa}}
\label{alg:RG-Armijo}
\end{algorithm}

\subsection{Geometric Means}
Intergrid transfer operators, used to transfer information between grids, are usually calculated using linear interpolation or injection. As interpolation requires computing averages, we need to generalize
such operations to our manifolds in order to devise e.g. prolongation operators.
Proposition \ref{prop:exp} enables to average points on $\mc{B}$ in a computationally convenient way.
\begin{proposition}\label{prop:geometric-mean}
Let $(\eta_{i})_{i\in I}\subset \mc{B}$ be arbitrary numbers and let 
\begin{equation}\label{eq:def-weights}
    \Omega = \{\w_{i}>0\colon i\in I\},\qquad
    \sum_{i\in I}\w_{i}=1
\end{equation}
be given corresponding weights. 
Then the weighted geometric mean of these numbers is
\begin{equation}\label{eq:def-mean-B}
    \ol{\eta}_{\Omega} := 
    \frac{\prod_{i\in I}\big(\frac{\eta_{i}}{1-\eta_{i}}\big)^{\w_{i}}}{1+\prod_{i\in I}\big(\frac{\eta_{i}}{1-\eta_{i}}\big)^{\w_{i}}}
    \in\mc{B}.
\end{equation}
Likewise, the weighted geometric mean of vectors $\{y^{i}\colon i\in I\}\subset\mc{B}_{n}$ is given by
\begin{equation}
    \ol{y}_{\Omega} = (\ol{y}_{j;\Omega})_{j\in I_{n}},
\end{equation}
that is, each component $(\ol{y}_{\Omega})_{j}$ is the weighted geometric average of the corresponding components $\{y^{i}_{j}\colon i\in I\}$ of the given vectors.
\end{proposition}
\begin{proof}
We define $\ol{\eta}_{\Omega}$ by the condition that characterizes the center of mass on a Riemannian manifold \cite[Lemma 6.9.4]{Jost:2017aa},
\begin{equation}\label{eq:mean-OC}
    0 = \sum_{i\in I}\w_{i}\exp_{\ol{\eta}_{\Omega}}^{-1}(\eta_{i}),
\end{equation}
except for using the exponential map \eqref{eq:def-exp} instead of the exponential map corresponding to the metric connection. Using \eqref{eq:def-inv-exp} we get
\begin{subequations}
\begin{align}
    0 &= \sum_{i\in I}\w_{i} \ol{\eta}_{\Omega}(1-\ol{\eta}_{\Omega})\log\frac{(1-\ol{\eta}_{\Omega})\eta_{i}}{\ol{\eta}_{\Omega}(1-\eta_{i})}.
    \intertext{Subdividing by $\ol{\eta}_{\Omega}(1-\ol{\eta}_{\Omega})\in\mc{B}$ and taking into account \eqref{eq:def-weights} gives \eqref{eq:def-mean-B},}
    \log\frac{1-\ol{\eta}_{\Omega}}{\ol{\eta}_{\Omega}}
    &= -\sum_{i\in I}\w_{i}\log\frac{\eta_{i}}{1-\eta_{i}}
    \\
    \ol{\eta}_{\Omega} 
    &= \frac{e^{\sum_{i\in I}\w_{i}\log\frac{\eta_{i}}{1-\eta_{i}}}}{1+e^{\sum_{i\in I}\w_{i}\log\frac{\eta_{i}}{1-\eta_{i}}}}
    = \frac{\prod_{i\in I}\big(\frac{\eta_{i}}{1-\eta_{i}}\big)^{\w_{i}}}{1+\prod_{i\in I}\big(\frac{\eta_{i}}{1-\eta_{i}}\big)^{\w_{i}}}.
\end{align}
\end{subequations}
\end{proof}
\begin{remark}[Geometric averaging in closed form]
We point out that the computation of Riemannian means typically requires to solve an optimality condition of the form \eqref{eq:mean-OC} by a fixed-point iteration. In the present case, the chosen geometry yields the corresponding geometric mean \eqref{eq:def-mean-B} in closed form.
\end{remark}

\subsection{Grid Transfer Operators}

\subsubsection{Prolongation}

We associate with every fine grid point $j\in I^{c}_{m}$ -- recall the notation and nomenclature in Section \ref{sec:Notation} -- a neighborhood $N_{j}$ and corresponding weights $\Omega_{j}$,
\begin{subequations}\label{eq:def-Nj}
\begin{align}
    N_{j} &\subset I_{m},\quad j\in I_{m}^{c}
\\
    \Omega_{j} &= \{\w_{i}\colon i\in N_{j}\}.
\end{align}
\end{subequations}
The prolongation map is defined by keeping coarse grid components and by assigning geometric averages of coarse grid components to fine grid components, i.e.
\begin{equation}\label{eq:def-P}
    P\colon\mc{B}_{m}\to\mc{B}_{n},\qquad
    P(x)_{j} = \begin{cases}
    x_{j} &\text{if}\; j\in I_{m}, \\
    \ol{x}_{\Omega_{j}}, &\text{if}\; j\in I_{m}^{c}.
    \end{cases}
\end{equation}
Tangent vectors are transferred from coarse to fine by the differential of the prolongation map.
\begin{lemma}\label{lem:dP}
The differential of the prolongation map \eqref{eq:def-P} is given by
\begin{subequations}\label{eq:dPxv}
\begin{align}
    dP_{x}\colon T_{x}\mc{B}_{m} &\to T_{P(x)}\mc{B}_{n},\\
    (dP_{x}u)_{j} &= \begin{cases}
    u_{j}, &\text{if}\; j\in I_{m}, \\
    \ol{x}_{\Omega_{j}}(1-\ol{x}_{\Omega_{j}})\sum_{i\in N_{j}}\frac{\w_{i}}{x_{i}(1-x_{i})} u_{i}, &\text{if}\; j\in I^{c}_{m}, \\
    0, &\text{otherwise.}
    \end{cases}
\end{align}
\end{subequations}
\end{lemma}
\begin{proof}
Put
\begin{subequations}
\begin{align}
    \kappa_{j}(x;\Omega)
    &= \sum_{i\in N_{j}}\w_{i}\log\frac{x_{i}}{1-x_{i}},
    \\ \label{eq:proof-mean-vphi}
    \vphi_{j}(x;\Omega)
    &= \ol{x}_{\Omega_{j}}
    \overset{\eqref{eq:def-mean-B}}{=}
    \frac{e^{\kappa_{j}(x;\Omega)}}{1+\kappa_{j}(x;\Omega)}.
\end{align}
\end{subequations}
Then
\begin{subequations}
\begin{align}
    d\kappa_{j}(x;\Omega)[u]
    &= \frac{d}{dt}\kappa_{j}(x+t u;\Omega)|_{t=0}
    = \sum_{i\in N_{j}}\frac{\w_{i}}{x_{i}(1-x_{i})} u_{i},
    \\
    d\vphi_{j}(x;\Omega)[u]
    &= \frac{e^{\kappa_{j}(x;\Omega)}}{(1+\kappa_{j}(x;\Omega))^{2}} d\kappa_{j}(x;\Omega)[u]
    \overset{\eqref{eq:proof-mean-vphi}}{=}
    \ol{x}_{\Omega_{j}}(1-\ol{x}_{\Omega_{j}})
    d\kappa_{j}(x;\Omega)[u]
\end{align}
\end{subequations}
and hence from \eqref{eq:def-P}
\begin{equation}
    (dP_{x}u)_{j}
    = (dP(x)[u])_{j}
    = \begin{cases}
    u_{j}, &\text{if}\,j\in I_{m}, \\
    d\vphi_{j}(x;\Omega)[u], &\text{if}\,j\in I_{m}^{c}, \\
    0, &\text{otherwise},
    \end{cases}
\end{equation}
which is \eqref{eq:dPxv}.
\end{proof}

\subsubsection{Restriction}
In view of Definition \ref{eq:def-P}, it is obvious that the only coarse grid vector $x$ that conforms to the fine grid vector $y=P(x)$ is the trivial restriction
\begin{equation}\label{eq:def-restriction}
    R\colon\mc{B}_{n}\to\mc{B}_{m},\qquad
    x = R y = [y]_{I_{m}} = (y_{j})_{j\in I_{m}}.
\end{equation}
Unlike the prolongation map $P$ \eqref{eq:def-P}, this restriction map $R$ is linear. Concerning the restriction of tangent vectors, using the differential $dR = R$ analogous to \eqref{eq:dPxv} would be a less sensible choice, however. Rather, we define for this purpose an individual operator as follows.

Let $x=R y$, $v\in T_{y}\mc{B}_{n}$ and $u\in T_{x}\mc{B}_{m}$. We define
\begin{subequations}\label{eq:def-TR}
\begin{align}
    TR_{y}\colon T_{y}\mc{B}_{n}&\to  T_{x}\mc{B}_{m},
    \\ \label{eq:def-TRy}
    TR_{y} &= G_{m}(x)^{-1} dP_{x}^{\T} G_{n}(y),\qquad x\in R y
    \\ \label{eq:def-TRyij}
    (TR_{y})_{ij}
    &= \begin{cases}
    (dP_{x}^{\T})_{ij}, &\text{if}\; j=i\in I_{m}
    \\
    \frac{x_{i}(1-x_{i})}{y_{j}(1-y_{j})}(dP_{x}^{\T})_{ij}, &\text{if}\; j\in I^{c}_{m}\;\text{and}\;i\in N_{j}
    \\
    0, &\text{otherwise.}
    \end{cases}, \\
    & \textnormal{ for }
    j\in I_{n}=I_{m}\dot\cup I_{m}^{c},\;
    i\in I_{m}.
\end{align}
\end{subequations}
The rationale behind this definition is the equation
\begin{equation}\label{eq:TR-dP-a}
    \la u, TR_{y} v\ra_{x} = \la dP_{x} u,v\ra_{y}
\end{equation}
that mimics in terms of $TR$ and the differential $dP$ the common Galerkin condition $R = P^{\T}$ in the case of Euclidean scenarios and linear mappings $R, P$, while taking into account the coarse  and find-grid metrics. Expanding this equation using \eqref{eq:metric-notation} gives
\begin{equation}\label{eq:TR-dP-b}
    \la u, G_{m}(x) TR_{y} v\ra = \la dP_{x} u,G_{n}(y) v\ra
\end{equation}
which implies \eqref{eq:def-TRy}. Inspecting the componentwise specification \eqref{eq:def-TRyij} of $TR_{y}$, we observe the aforementioned relation $TR_{y}=dP_{x}^{\T}$ in the case of coinciding coarse and fine grid points indexed by $i$ and $j$, whereas the second line of \eqref{eq:def-TRyij} shows that vector transport from the fine grid to the coarse grid is involved -- cf.~Remark \ref{rem:vector-transport} and \eqref{eq:def-Piuv} -- if $j\not\in I_{m}^{c}\subset I_{n}$ differs from $i\in I_{m}$.

As a result, pushing tangent vectors to the coarse grid using the nonlinear mapping $TR$ better conforms to the nonlinear prolongation map \eqref{eq:def-P}, than merely using $dR = R$ based on the restriction map \eqref{eq:def-restriction}, which is a reasonable choice for transferring vectors of $\mc{B}_{n}$ (rather than tangent vectors of $T\mc{B}_{n}$).

\subsection{Two-Level Geometric Optimization}\label{sec:two-level-geometric}
We generalize the two-level Euclidean optimization approach of Section \ref{sec:two-level-Euclidean-optimization} to our Riemannian setting.

\subsubsection{Coarse Objective Function}
We adapt the coarse grid model \eqref{eq:def-psi-Euclidean} of the objective function at $x_{0}=R y_{0}$ to the present geometric setting,
\begin{subequations}\label{eq:def-psi-coarse-grid}
\begin{align}
    \psi(x) &= f(x)-\la\exp_{x_{0}}^{-1}(x),\kappa(x_{0},y_{0})\ra_{x_{0}},
    \label{eq:def-psi-coarse-grid-a} \\ \label{eq:def-psi-coarse-grid-b}
    \kappa(x_{0},y_{0}) &= \ggrad f(x_{0}) - TR_{y_{0}}\ggrad f(y_{0}),
\end{align}
\end{subequations}
with the Riemannian gradients $\ggrad f(x_{0})=G_{m}(x_{0})^{-1}\partial f(x_{0})$, $\ggrad f(y_{0})=G_{n}(y_{0})^{-1}\partial f(y_{0})$. We check the first-order coherence condition corresponding to \eqref{eq:first-order-coherence}.
\begin{lemma}[First-order coherence condition]
The coarse grid model \eqref{eq:def-psi-coarse-grid} satisfies
\begin{equation}\label{eq:cond-1st-geom}
    \ggrad\psi(x_{0}) = TR_{y_{0}}\ggrad f(y_{0}).
\end{equation}
\end{lemma}
\begin{proof}
By Lemma \ref{eq:dexp-dexp-1}, we have
\begin{equation}
    \frac{\partial}{\partial x_{i}} \la\exp_{x_{0}}^{-1}(x),\kappa(x_{0},y_{0})\ra
    = \frac{x_{0i}(1-x_{0i})}{x_{i}(1-x_{i})} \kappa_{i}(x_{0},y_{0})
\end{equation}
and hence
\begin{subequations}
\begin{align}
    \partial \la\exp_{x_{0}}^{-1}(x),\kappa(x_{0},y_{0})\ra_{x_{0}}
    &= G_{m}(x)G_{m}(x_{0})^{-1} G_{m}(x_{0})\kappa(x_{0},y_{0}) \\
    &= G_{m}(x)\kappa(x_{0},y_{0}).
\end{align}
\end{subequations}
Substituting into \eqref{eq:def-psi-coarse-grid} gives
\begin{subequations}
\begin{align}
    \ggrad\psi(x) &= \ggrad f(x)-\kappa(x_{0},y_{0})
    \\
    &\overset{\eqref{eq:def-psi-coarse-grid-b}}{=} 
    \ggrad f(x) - \ggrad f(x_{0}) + TR_{y_{0}}\ggrad f(y_{0}).
\end{align}
\end{subequations}
Setting $x=x_{0}$ yields \eqref{eq:cond-1st-geom}.
\end{proof}

\subsubsection{Coarse Grid Descent Step, Algorithm}
We examine the effect of minimizing the coarse grid function \eqref{eq:def-psi-coarse-grid}, analogous to Section \ref{sec:Coarse-descent-direction}. Rearranging terms, we have
\begin{subequations}
\begin{align}
    \psi(x) - f(x_{0}) &= f(x)-f(x_{0}) - \la \exp_{x_{0}}^{-1}(x),\ggrad f(x_{0})\ra_{x_{0}} \\
    & \quad + \la \exp_{x_{0}}^{-1}(x), TR_{y_{0}}\ggrad f(y_{0}) \ra_{x_{0}}
    \\
    &\overset{\eqref{eq:TR-dP-a}}{=}
    f(x)-f(x_{0}) - \la \exp_{x_{0}}^{-1}(x),\partial f(x_{0})\ra \\
    & \qquad + \la dP_{x_{0}}\exp_{x_{0}}^{-1}(x),\ggrad f(y_{0}) \ra_{y_{0}}.
\end{align}
\end{subequations}
Following the reasoning of Section \ref{sec:Coarse-descent-direction} shows that, if $x$ can be determined such that $\psi(x)< f(x_{0})$, then $dP_{x_{0}}\exp_{x_{0}}^{-1}(x)$ is a descent direction with respect to $f$ on the \textit{fine} grid if the first three terms on the right-hand side, i.e.~$f(x)-f(x_{0}) - \la \exp_{x_{0}}^{-1}(x),\partial f(x_{0})\ra$, are nonnegative. 

In the Euclidean case, with $x-x_{0}$ in place of $\exp_{x_{0}}^{-1}(x)$, the latter automatically holds since $f(x)-f(x_{0}) - \la \exp_{x_{0}}^{-1}(x),\partial f(x_{0})\ra = D_{f}(x,x_{0})\geq 0$ (cf.~Remark \ref{rem:psi-Df}). In the present geometric setting, however, this reasoning can only be guaranteed if $\|x-x_{0}\|$ is sufficiently small, as stated in the following Lemma.
\begin{lemma}\label{lem:exp-x0-1}
For the mapping defined by \eqref{eq:def-inv-exp}, one has
\begin{equation}\label{eq:exp-1-Taylor}
    \exp_{x_{0}}^{-1}(x) = x-x_{0} + \mc{R}(x-x_{0}) \quad\text{with}\quad
    \lim_{x\to x_{0}}\frac{\mc{R}(x-x_{0})}{\|x-x_{0}\|} = 0.
\end{equation}
Furthermore,
\begin{equation}\label{eq:exp-1-sign}
    (x-x_{0})_{i}\big(\exp_{x_{0}}^{-1}(x)\big)_{i} \geq 0,\qquad \forall i\in I_{m}.
\end{equation}
\end{lemma}
\begin{proof}
The expansion \eqref{eq:exp-1-Taylor} is easy to compute. As for \eqref{eq:exp-1-sign}, we have by \eqref{eq:dinvexp} that, for each component $i\in I_{m}$, the derivative of the function $x_{i}\mapsto \exp_{x_{0i}}(x_{i})$ is nonnegative. Consequently, this function is monotonously increasing as is the function $x_{i}\mapsto x_{i}-x_{0i}$, and both functions only attain the value $0$ at $x_{0i}$.
\end{proof}
Expansion \eqref{eq:exp-1-Taylor} shows that for $x$ sufficiently close to $x_{0}$, $dP_{x_{0}}\exp_{x_{0}}^{-1}(x)$ is a descent direction on the fine grid whenever $\psi(x)<f(x_{0})$. Relation \eqref{eq:exp-1-sign}, on the other hand, shows that the vector $\exp_{x_{0}}^{-1}(x)$ points into the positive half space defined by $x-x_{0}$ and, in this sense, is aligned with $x-x_{0}$. Yet, for increasing $\|x-x_{0}\|$, the nonlinearity of this term increases as well, such that the inner product $\la \exp_{x_{0}}^{-1}(x),\partial f(x_{0})\ra$ may not allow to certify a descent direction by merely checking the sign of $\psi(x)-f(x_{0})$. Therefore, besides the geometric version of condition \eqref{coarse_condition},
\begin{equation}\label{eq:Gratton-geometric}
    \|TR_{y_{0}}\ggrad f(y_{0})\|_{R y_{0}} \geq \eta \|\ggrad f(y_{0})\|_{R y_{0}}
    \qquad\wedge\qquad
    \|y-y_{c}\| \geq \veps,
\end{equation}
the following algorithm involves the additional test 
\begin{equation}\label{eq:Df-geom-test}
    \wt{D}_{f}(x,x_{0}):= f(x)-f(x_{0}) - \la \exp_{x_{0}}^{-1}(x),\partial f(x_{0})\ra \geq 0,
\end{equation}
that still can be cheaply done at the coarse level, rather than checking directly and more costly the vector $dP_{x_{0}}\exp_{x_{0}}^{-1}(x)$ on the fine level.

\begin{algorithm}[ht]
\DontPrintSemicolon
\textbf{initialization:} initial fine grid point $y$, corresponding coarse grid point $x=R y$.\;
\Repeat{termination criterion holds.}{
  set $y_{0}=y$, $x_{0}=R y$; \;
  \If{condition \eqref{eq:Gratton-geometric}
 is satisfied at $x_{0}$}{
    find $x$ such that $\psi(x)<f(x_{0})$; \tcc*{coarse grid update} 
    \If{condition \eqref{eq:Df-geom-test} holds}{  
    set $h = dP_{x_{0}}\exp_{x_{0}}^{-1}(x)$; \;
    find $\alpha>0 \colon f(\exp_{y_{0}}(\alpha h)) < f(y_{0})$;\tcc*{line search} 
    $y \leftarrow \exp_{y_{0}}(\alpha h)$ \tcc*{update}
    }
  }\Else{
  apply one fine level iteration to compute $y$ such that $f(y) < f(y_{0})$. \tcc*{update}
  }
  }
\caption{Two-Level Geometric Optimization\label{alg:twogridgeometric}}
\end{algorithm}

\subsection{MultiLevel Extension}\label{sec:multilevel-geometric}
Analogous to Section \ref{sec:multilevel-Euclidean}, we introduce the notation
\begin{subequations}
\begin{align}
    x^{(\ell)} &\in \mc{B}^{(\ell)}_{n_{\ell}}
    \intertext{for the variables at levels $\ell\in\{0,1,2,\dotsc\}$, where} 
    x^{(0)} = x &\in \mc{B}_{n} = \mc{B}^{(0)}_{n_{0}}
\end{align}
are the original fine grid variables. The components $x^{(\ell)}_{i}$ of these vectors are indexed by the nested index sets
\begin{equation}
    i\in I_{n_{\ell}} \subset I_{n_{\ell-1}} \subset \dotsb \subset I_{n} = I_{n_{0}}.
\end{equation}
\end{subequations}
Accordingly, the coarse grid model \eqref{eq:def-psi-coarse-grid} applied at level $\ell$ reads
\begin{subequations}
\begin{align}
    \psi(x^{(\ell)}) &= f(x^{(\ell)})-\la\exp_{x_{0}^{(\ell)}}^{-1}(x^{(\ell)}),\kappa(x_{0}^{(\ell)},x_{0}^{(0)})\ra_{x_{0}^{(\ell)}},
    \\
    \kappa(x_{0}^{(\ell)},x_{0}^{(0)}) &= \ggrad f(x_{0}^{(\ell)}) - TR_{(\ell)}^{(0)}(x_{0}^{(0)})\ggrad f(x_{0}^{(0)}),
\end{align}
\end{subequations}
with $x_{0}^{(\ell)} = R_{(\ell)}^{(0)} x_{0}^{(0)}$. Prolongation operators $P_{(0)}^{(\ell)},\,\ell>1$ can be defined in the same way as $P_{0}^{(1)}=P$ is defined by \eqref{eq:def-P}, with increasingly larger neighborhoods \eqref{eq:def-Nj}. By this, the mappings $TR_{(\ell)}^{(0)}$ for pushing tangent vectors from level $0$ to $\ell$ follow analogous to \eqref{eq:TR-dP-a}.

This formulation adapts to the geometric setting the Euclidean multilevel approach outlined in Section \ref{sec:multilevel-Euclidean}, and similar comments concerning the trade-off of invoking coarse grid models at various levels apply. We leave a detailed study of this multilevel extension for future work.

\newcommand{\showPhantoms}[1]{
\begin{centering}
		\begin{tabular}{c@{\hskip 0.4em}c@{\hskip 0.4em}c@{\hskip 0.4em}c@{\hskip 0.4em}c@{\hskip 0.4em}c}
		
			\includegraphics[width=#1\textwidth]{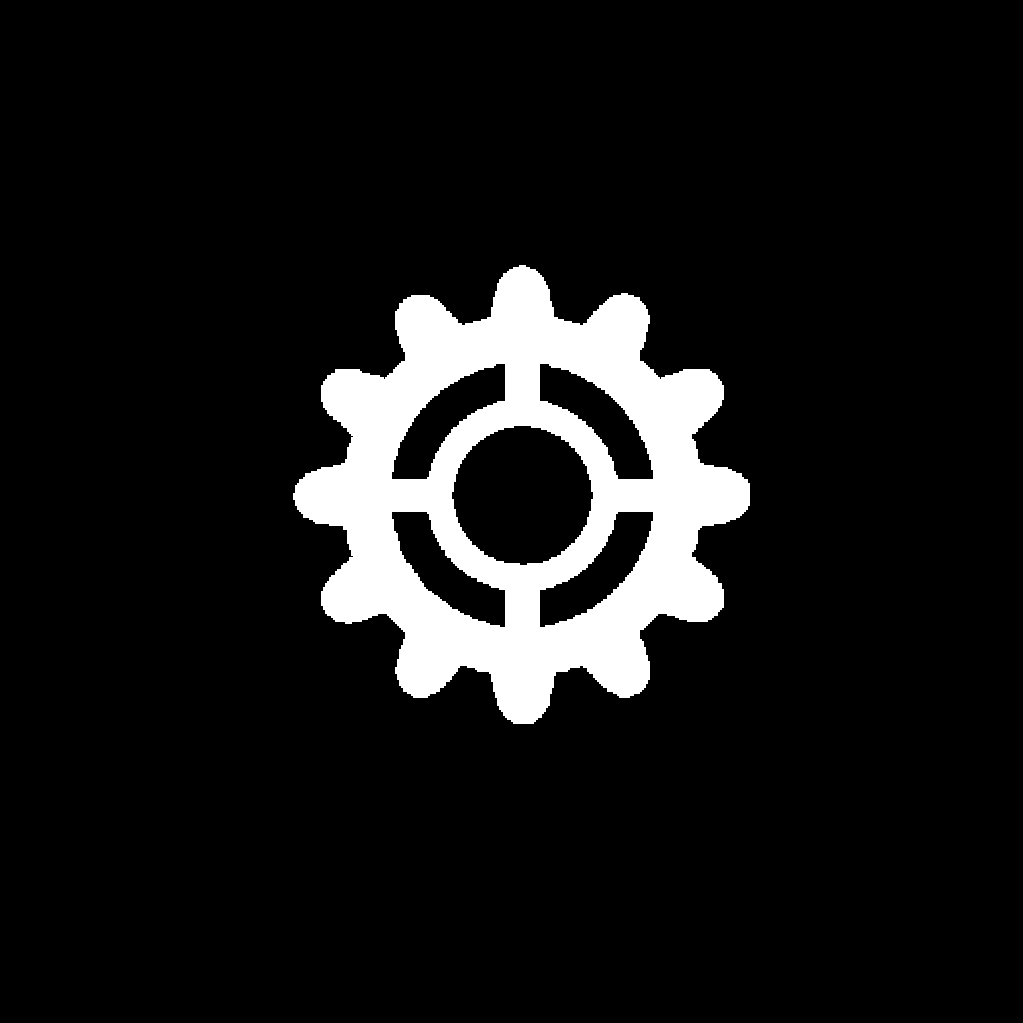}
			&
			\includegraphics[width=#1\textwidth]{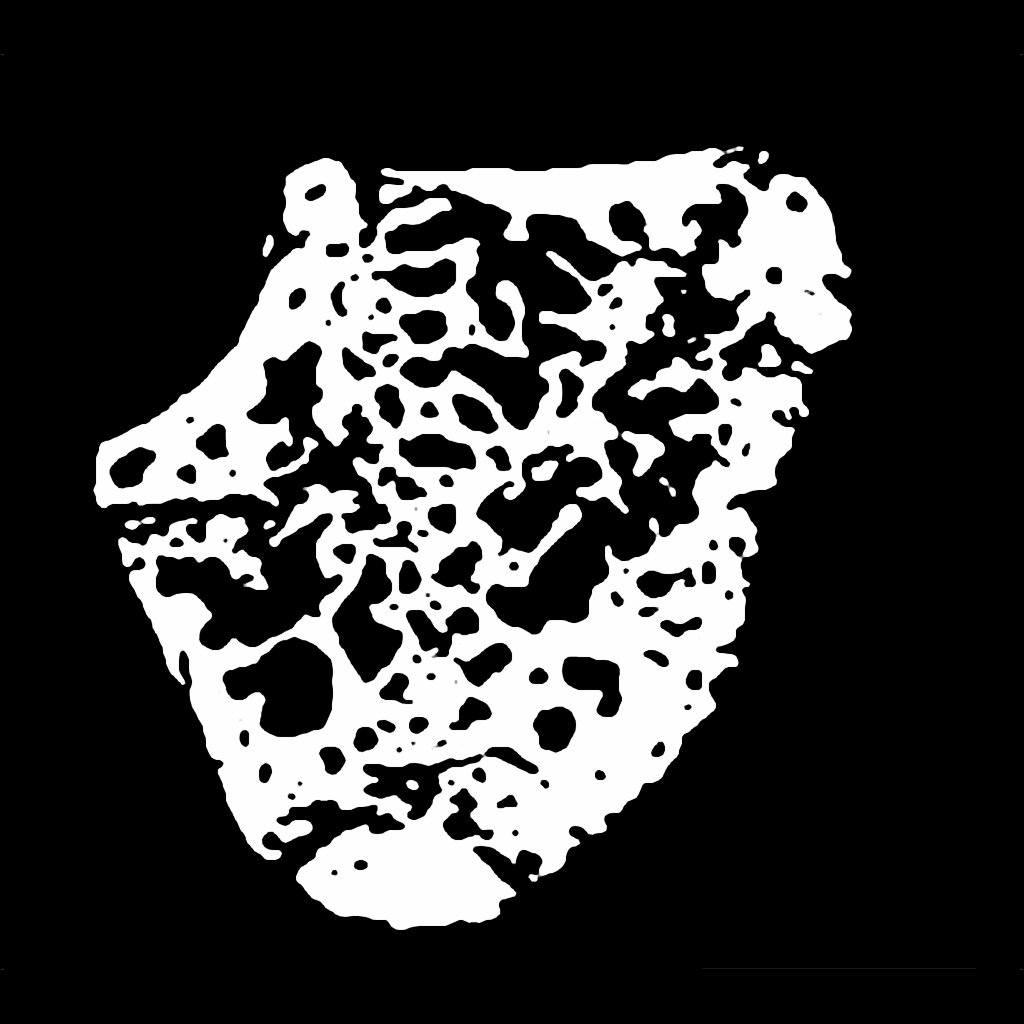}
			&
			\includegraphics[width=#1\textwidth]{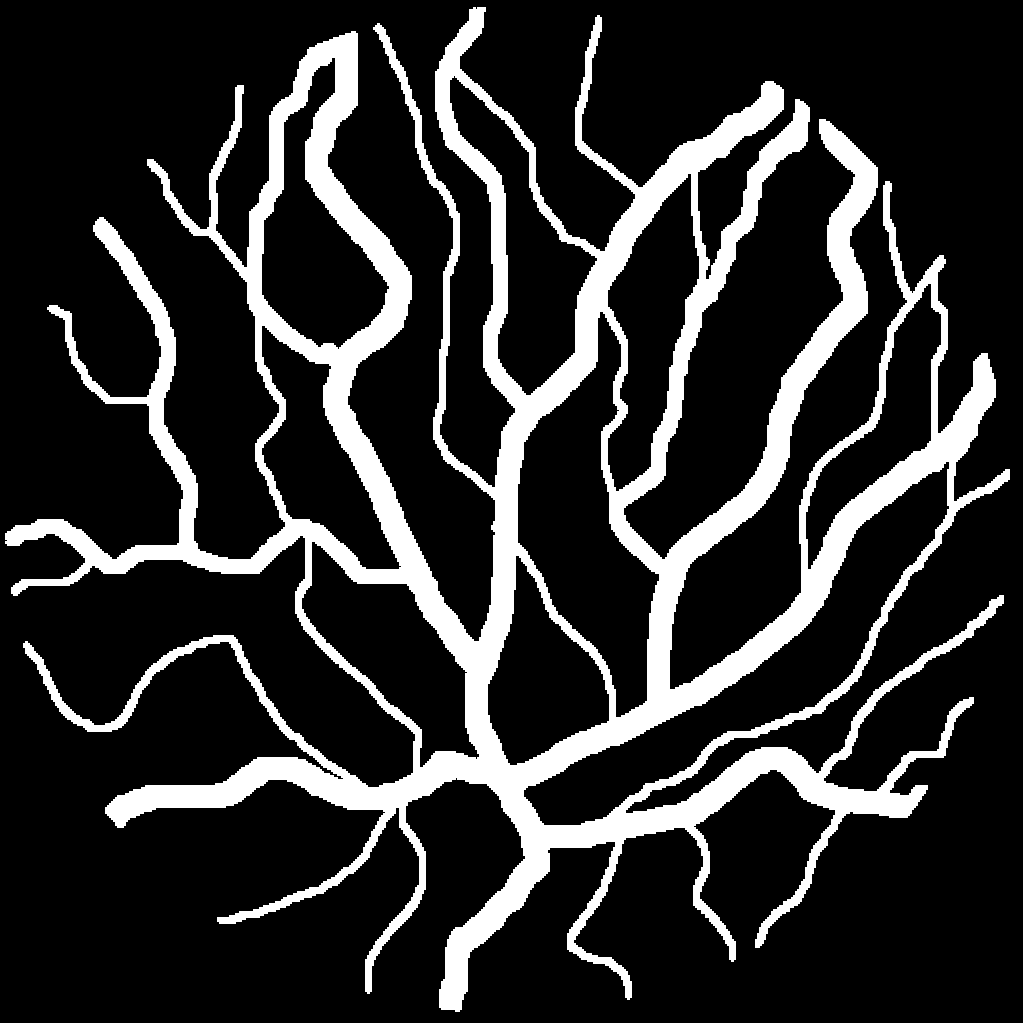}
			&
			\includegraphics[width=#1\textwidth]{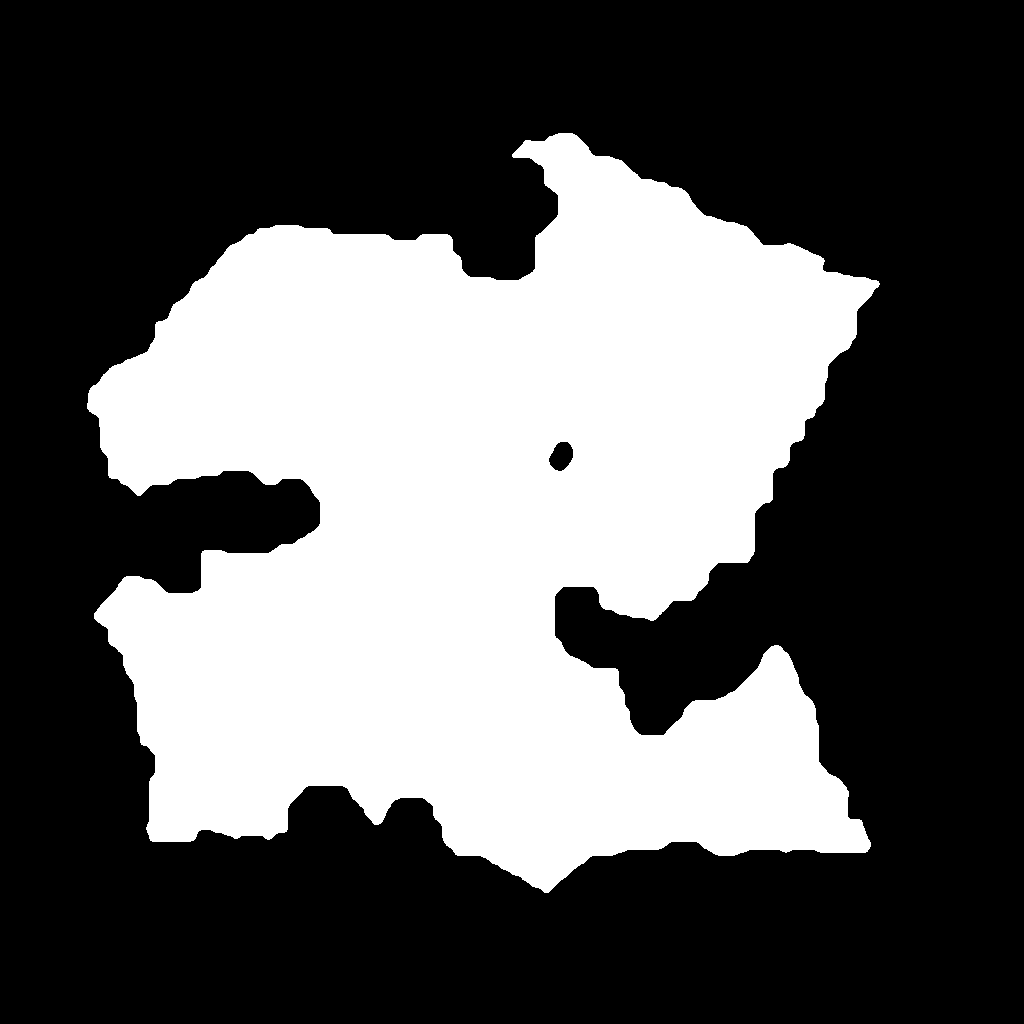}
			&
			\includegraphics[width=#1\textwidth]{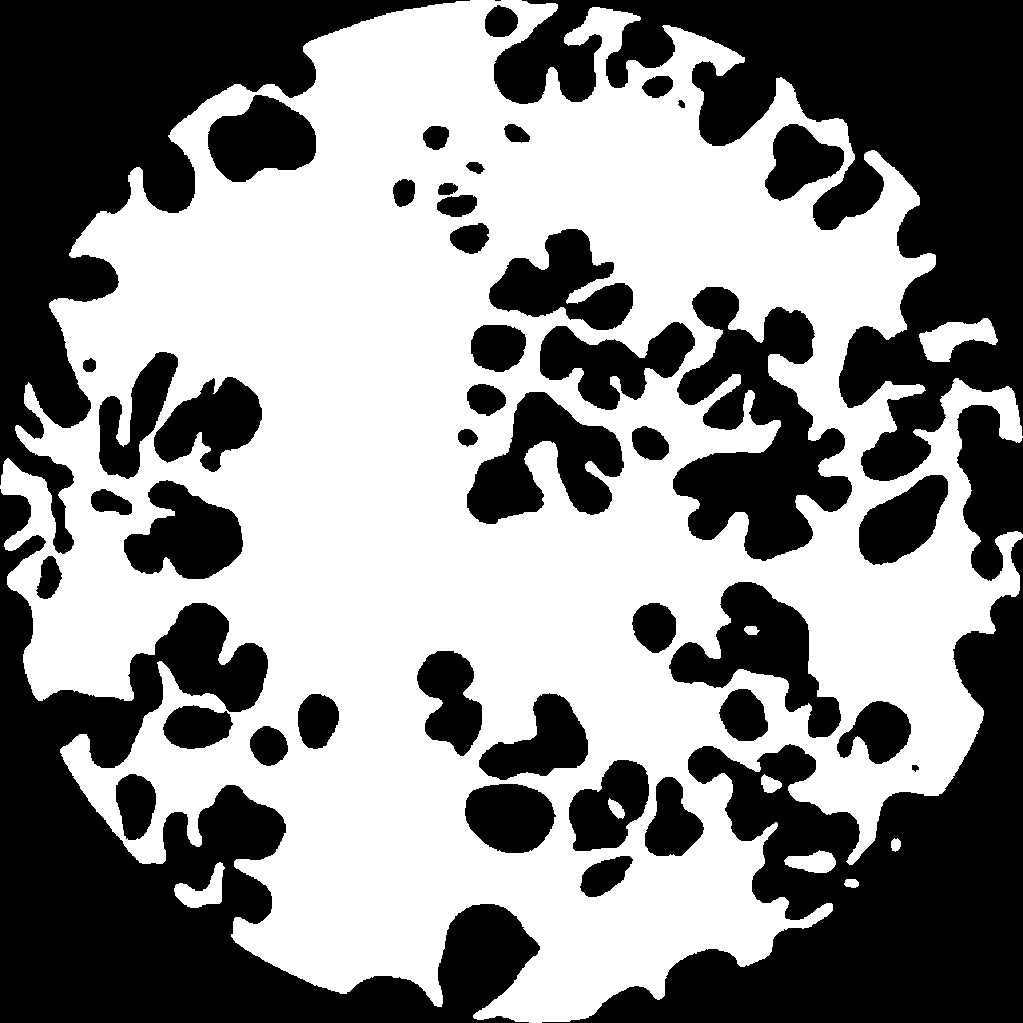}
			&
			\includegraphics[width=#1\textwidth]{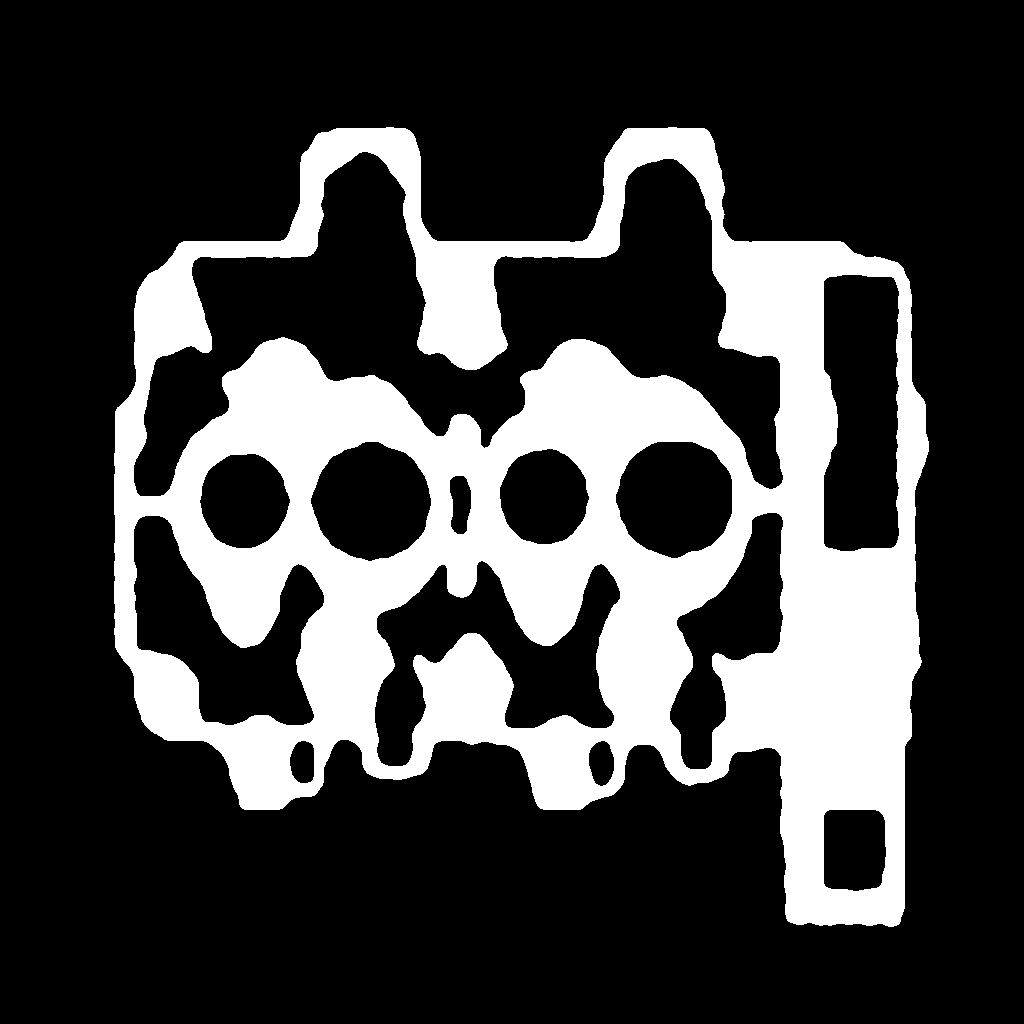}
			 \\
			\footnotesize \texttt{gear} & \footnotesize \texttt{bone} & \footnotesize \texttt{vessel} & \footnotesize \texttt{batenburg} & \footnotesize \texttt{roux} & \footnotesize \texttt{skulls} 
			
		\end{tabular}
	\end{centering}
}

\section{Experiments}\label{sec:Experiments}
In this section we numerically evaluate the proposed approach. In a first experiment we compare the Riemannian Gradient (RG) descent, see Algorithm ~\ref{alg:RG-Armijo},
with a state-of-the-art Accelerated Bregmann Proximal Gradient (ABPG) method \cite{Hanzely:2021vc}. In a second experiment we evaluate one-level vs. two-level (2L) schemes and compare our approach using Riemannian Gradient descent on both fine and coarse grids with a state-of-the-art first-order multilevel approach~\cite{Kocvara:2016} based on Projected Gradients (PG).

\subsection{Data Setup}
For our numerical evaluation we considered the phantom ($1024\times1024$ pixels) shown by Figure \ref{fig:phantoms}. The tomographic projection matrices were computed by means of the  ASTRA-toolbox\footnote{https://www.astra-toolbox.com}, using parallel beam geometry with equidistant angles in the range $[0, \pi]$. The undersampling rate at the fine grid was chosen to be $2\%$ and $20\%$, respectively, corresponding to $20$ and $200$ projection angles. Each entry $a_{ij}$ of
the projection matrix $A$ corresponds to the length of the line segment of the $i$-th projection ray passing through the $j$-th pixel in the image domain (Figure \ref{fig:find-coarse-A}). At every level the width of the detector-array was set to the grid size, so that at each scale every pixel intersects with at least a single projection ray.

\begin{figure}[ht]
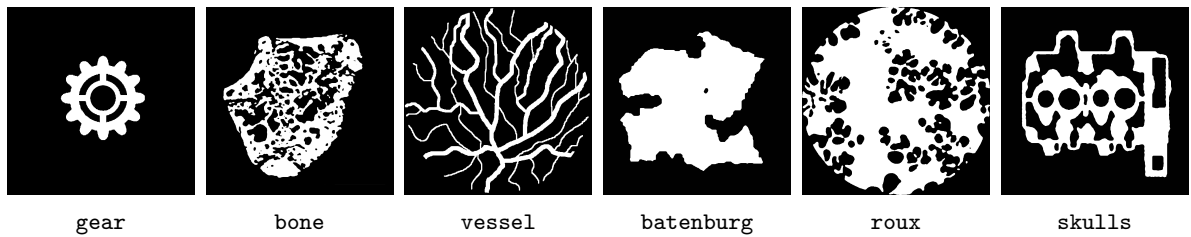

    \showPhantoms{0.15}
    
  \caption{The phantoms ($1024\times1024$) used for the numerical evaluation exhibit both fine and large scale structures and various shapes.}
  \label{fig:phantoms}
\end{figure}

\subsection{Implementation Details}
The coarse grid update in Algorithm \ref{alg:twogridgeometric} was performed by Armijo line search with a modification, such that if the calculated feasible stepsize was numerically close to $0$, then the coarse step was discarded and a normal step was performed instead. 

For an efficient implementation of the geometric grid transfer operators we used vectorized versions. For the experiments in this work we use uniform weights $\omega_i$, therefore the geometric mean, defined in Eq. \eqref{eq:def-mean-B}, can be reformulated as
    \begin{equation}
    \ol{\eta}_{\Omega} 
    =  \frac{\prod_{i\in I}\big(\frac{\eta_{i}}{1-\eta_{i}}\big)^{\w_{i}}}{1+\prod_{i\in I}\big(\frac{\eta_{i}}{1-\eta_{i}}\big)^{\w_{i}}}
    = \frac{e^{\sum_{i\in I}\w_{i}\log\frac{\eta_{i}}{1-\eta_{i}}}}{1+e^{\sum_{i\in I}\w_{i}\log\frac{\eta_{i}}{1-\eta_{i}}}}
    = \frac{e^{BI\big(\log\frac{\eta}{1-\eta}\big)}}{1+e^{BI\big(\log\frac{\eta}{1-\eta}\big)}},
\end{equation}
where $BI(v)$ is the standard bilinear interpolation operator. We used  PyLops\footnote{https://github.com/PyLops/pylops}, the linear operator library for Python, for performing bilinear interpolation and its transpose operation. Consequently, the grid transfer operators \eqref{eq:def-P}, \eqref{eq:dPxv} and \eqref{eq:def-TR} can be vectorized by using bilinear interpolation, resulting in
\begin{subequations}
\begin{align}
    P(x) &= \ol{x}_{\Omega} \\
    dP_{x}u &= G_{n}(P(x))^{-1} BI(G_{m}(x)u) \\
    TR_{y}v & = BI^{\T}(G_{n}(P(R(y)))^{-1} G_{n}(y) v),
\end{align}
\end{subequations}
where $BI^{\T}(v)$ denotes the transposed bilinear interpolation operator. Note that in case of non-uniform weights, the operation $BI(v)$ can be still represented by a corresponding linear operator.

In order to avoid numerical issues very close to the boundary of the box manifold, we clip each component $y_i$ to $[\varepsilon, 1- \varepsilon]$ with $\varepsilon = 10^{-10}$. This clipping step is performed directly after evaluating the exponential map \eqref{eq:def-exp}.

For all algorithms in our numerical evaluation we used the same set of parameters. The maximum number of iterations was set to $50$, which also serves as termination criterion. We used the natural initialization at $y^{0} = 0.5\, \eins$ (uninformative choice). The parameters for the Armijo~line~search~\cite[Def. 4.2.2]{Absil:2008aa} were set to $\sigma = 10^{-4}$, $\beta = 0.6$, $\alpha_0 = 1 / \beta$. For condition \eqref{coarse_condition} and its geometric version \eqref{eq:Gratton-geometric} for optimization on the coarse model, we selected $\eta = 0.49$ and $\varepsilon = 10^{-3}$. Regarding the objective function $f$ defined in Eq. \eqref{eq:intro-approach}, we set $\lambda=0.5$ for weighting the smoothed total variation with smoothing parameter $\rho=0.5$.

\subsection{Experiments and Discussion}
In a first experiment, we studied the impact of the Riemannian geometry on first-order optimization. We compared Riemannian gradient descent with a state-of-the-art accelerated proximal gradient method, the ABPG method \cite[Alg. 4]{Hanzely:2021vc}, 
that is suited to minimize objectives that are not Lipschitz-smooth, like the $\KL$-divergence based objective function \eqref{eq:intro-approach}. Results are shown as Figure \ref{fig:RGvsABPG}. By choosing the two quite different under-sampling scenarios, we created one problem instance with 200 angles (20\% undersampling) that is likely to have a unique optimum, see \cite{Kuske2019}, and a second problem with
multiple solutions due to the high undersampling ratio of 2\% with 20 angles. 
The proposed method performed on par or better than ABPG and was significantly superior for that latter three phantoms shown in Figure \ref{fig:phantoms} that have more homogeneous shapes.

In a second experiment we studied the influence of a using a second coarse level for computing efficiently descent directions on a $512\times 512$ grid, as summarized by Algorithm \ref{alg:twogridgeometric}. Figures \ref{fig:2LevelResults} and \ref{fig:2LevelResults-20angles} show the results for the Riemannian gradient descent: using one level versus two levels compared with state of the art ABPG. We observed the following. Firstly, the impact of the two-level acceleration is much stronger in the highly underdetermined case. Secondly, we concede 
that although ABPG performed worse in both scenarios, the cost of the iterations on the fine level 
are more expensive when using our method due to line search. On the coarser level however, the higher cost of the line search is compensated by using a much smaller number of variables for computing search directions.

Finally, we compared the proposed method to the Euclidean state-of-the-art multilevel method in \cite{Kocvara:2016}, that is capable on handling box constraints. The results are shown as Figure \ref{fig:2LevelResults-20anglesPGvsRG}. We call the latter method two-level projected gradient (2L PG)
as it uses the projected gradient method on both levels for optimizing the fine grid objective as well as the coarse grid model. Two-level projected gradient uses a scaled-down box that is adapted to the current iterate in such a way that feasibility is achieved after prolongation of the coarse level correction.
Such an adaption of the constraints is not needed in our proposed approach and therefore makes our approach more flexible. We notice however that the impact of the second level in the case of the projected gradient is stronger than in the geometric case.
We believe that this is due to the fact that our coarse model looses its convexity when the iterate is close to the boundary of the feasible set.

\begin{figure}[ht]
 \centering
   \includegraphics[width=0.45\textwidth]{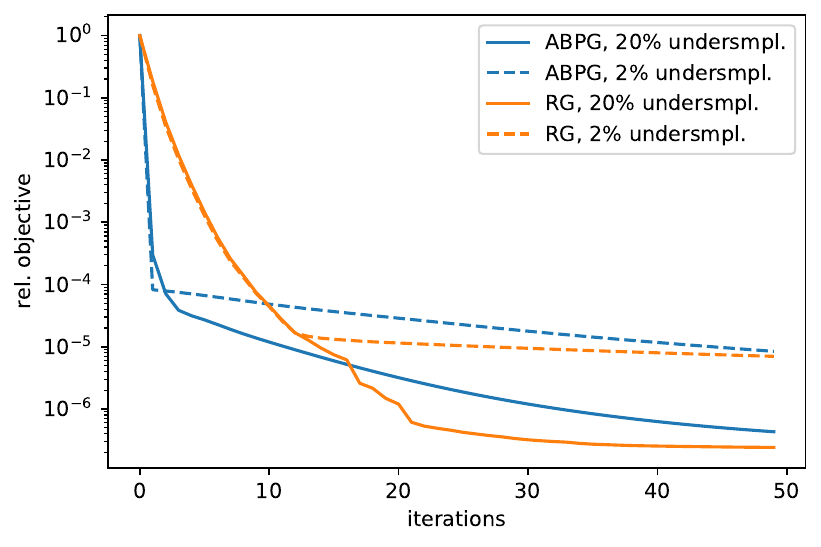}
   \includegraphics[width=0.45\textwidth]{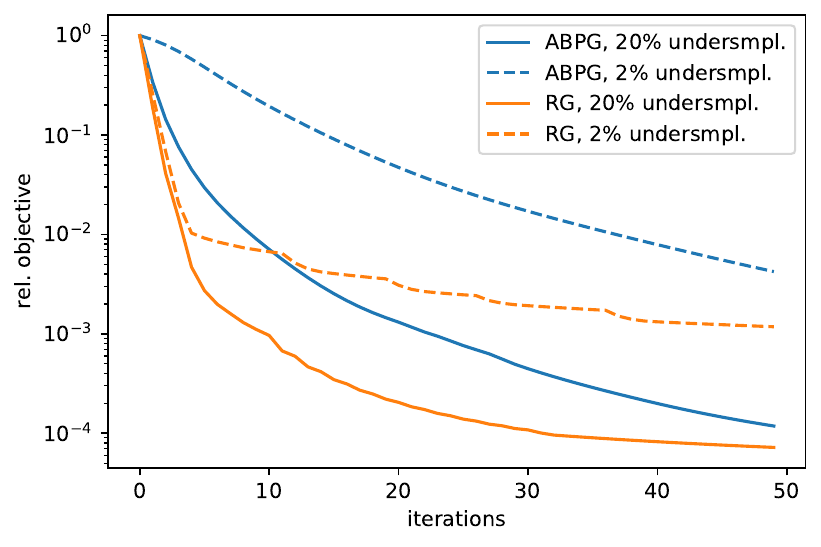}\\
 \includegraphics[width=0.45\textwidth]{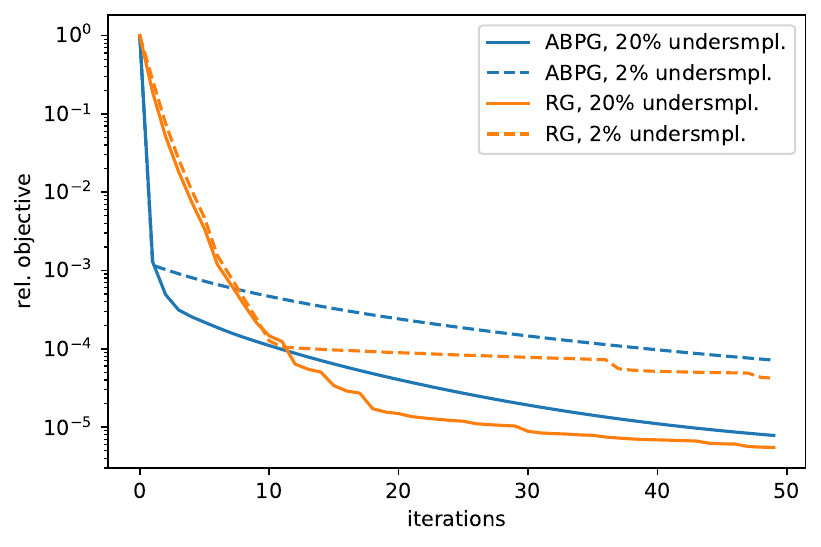}
 \includegraphics[width=0.45\textwidth]{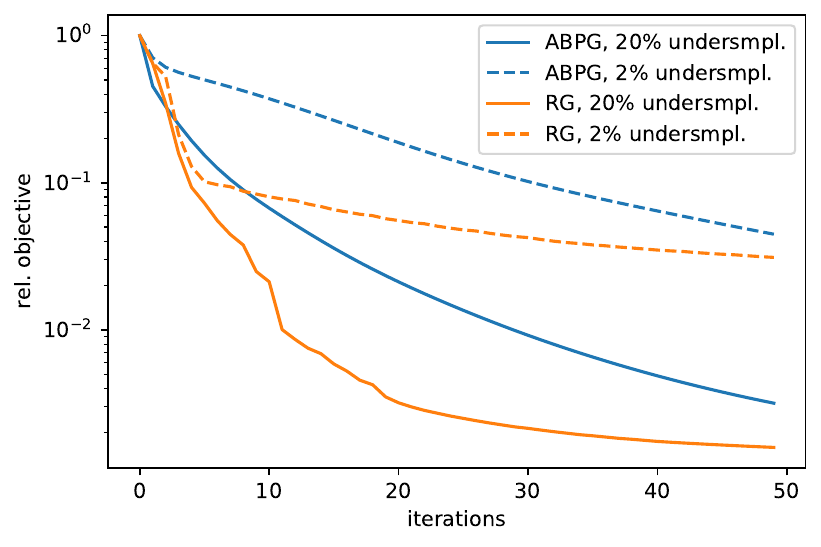}\\
 \includegraphics[width=0.45\textwidth]{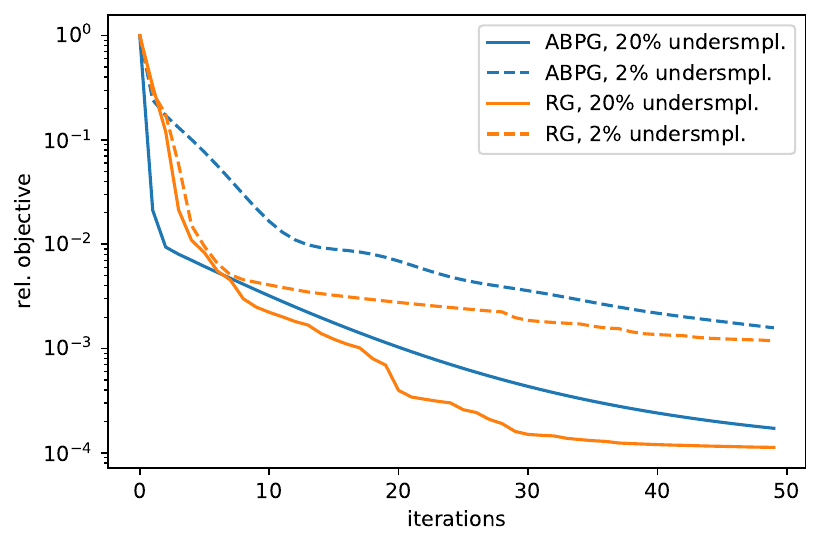}
 \includegraphics[width=0.45\textwidth]{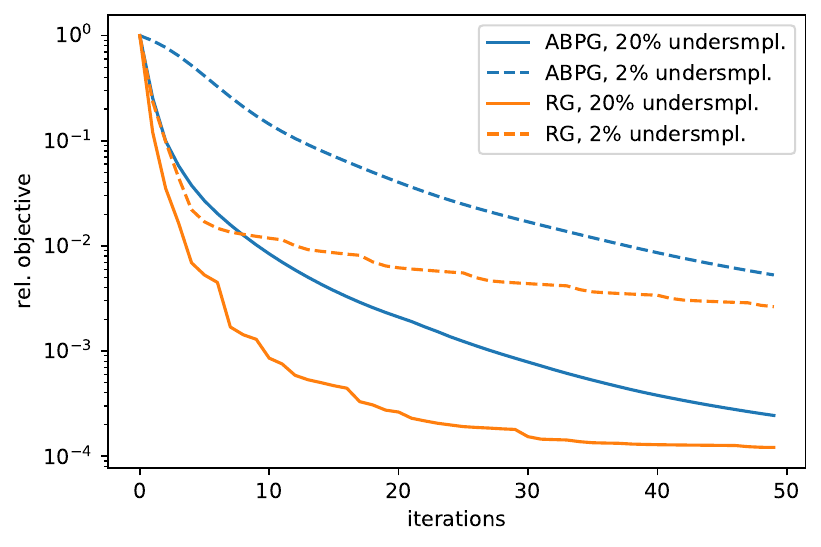}
  \caption{
  \textbf{ABPG vs. Riemannian gradient (RG) descent} in terms of relative objective values is compared on 2\% and 20\% undersampled projection data. The left column corresponds to the first three phantoms: \texttt{gear}, \texttt{bone}, \texttt{vessel}  in Figure \ref{fig:phantoms}, the right columns to the last three:  \texttt{batenburg}, \texttt{roux}, \texttt{skulls}. The Riemannian gradient descent achives lower objective values due to proper step-size selection (line search). For the more homogeneous phantoms corresponding to the right column, the proposed method significantly outperforms ABPG~\cite{Hanzely:2021vc}.
  }
  \label{fig:RGvsABPG}
 \end{figure}
 \begin{figure}[ht]
 \centering
   \includegraphics[width=0.45\textwidth]{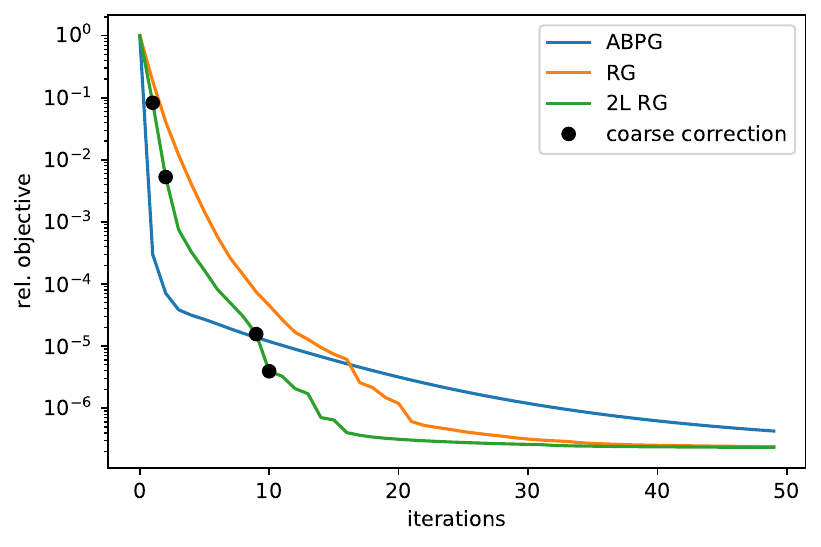}
   \includegraphics[width=0.45\textwidth]{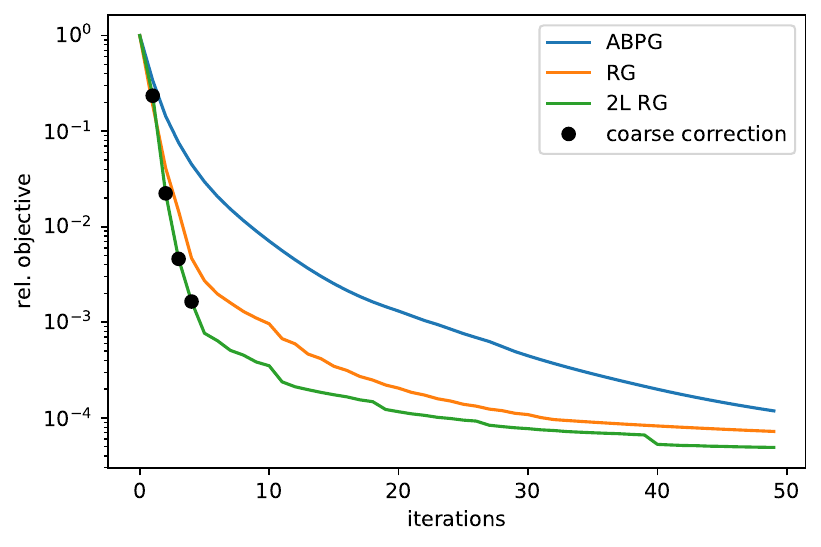}\\
 \includegraphics[width=0.45\textwidth]{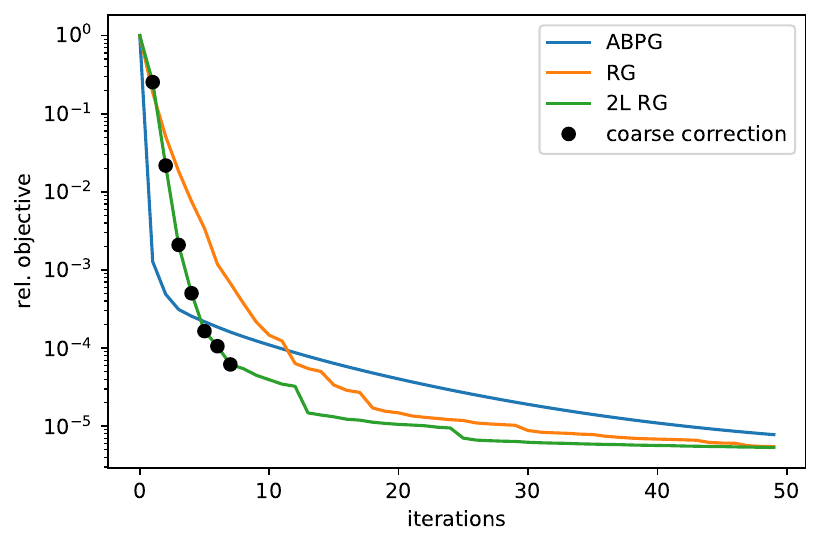}
 \includegraphics[width=0.45\textwidth]{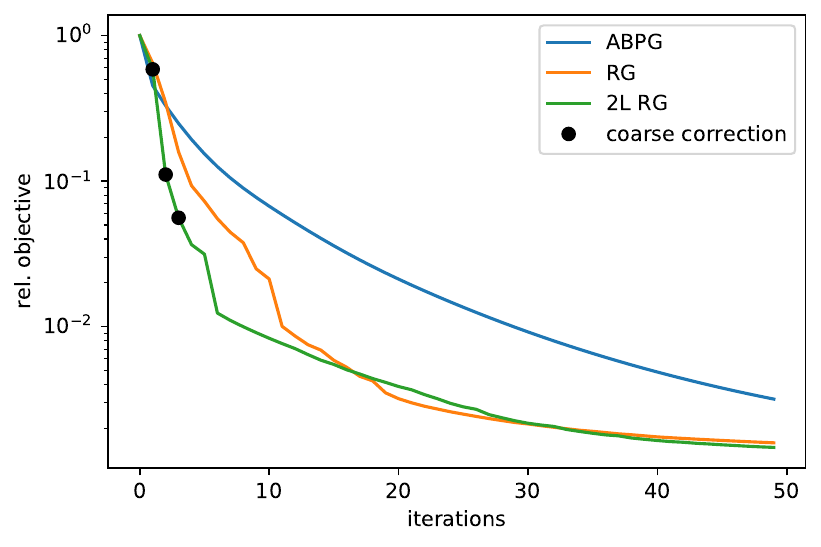}\\
 \includegraphics[width=0.45\textwidth]{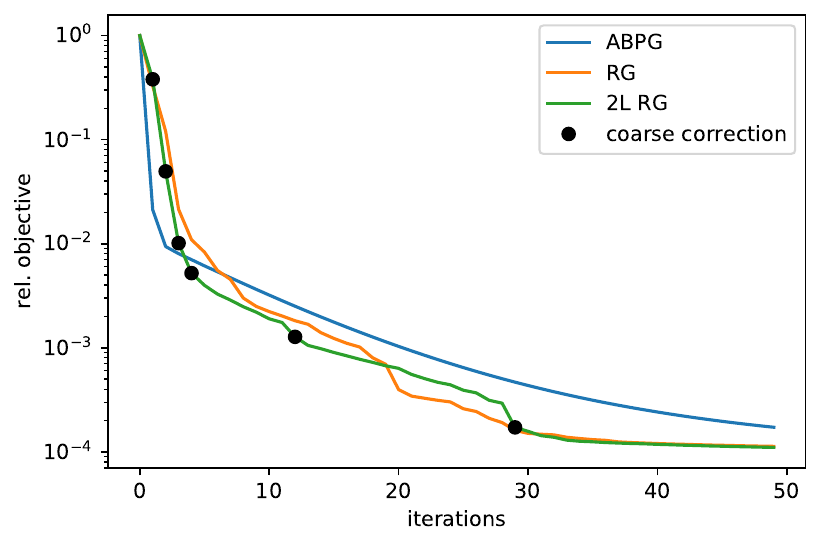}
 \includegraphics[width=0.45\textwidth]{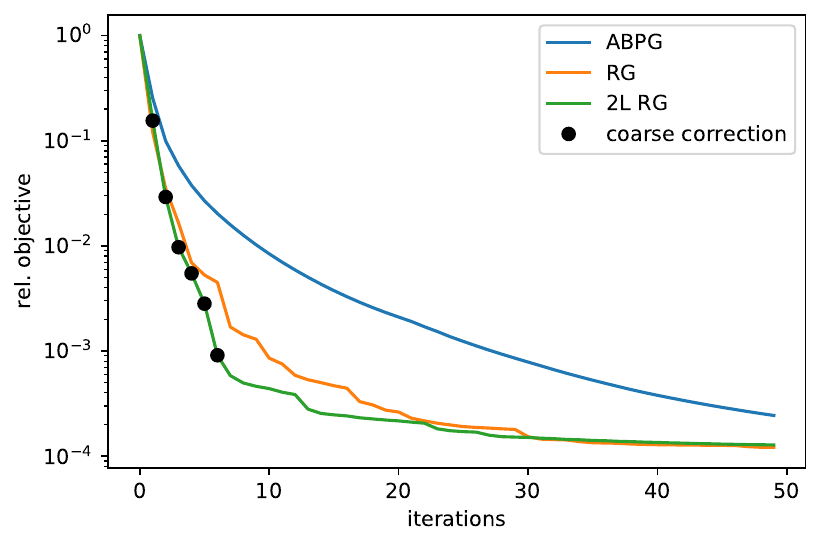}
  \caption{
  \textbf{20\% undersampling: comparison of relative objective function values} for single-level resp.~two-level (2L) Riemannian gradient descent (RG, 2L RG) and ABPG \cite{Hanzely:2021vc}. 
  The left column corresponds to the first three phantoms: \texttt{gear}, \texttt{bone}, \texttt{vessel} in 
  Figure \ref{fig:phantoms}, the right columns to the last three: \texttt{batenburg}, \texttt{roux}, \texttt{skulls}.
  Black dots indicate when descent directions were computed on coarser grids. 
  On the coarse grid we have 40\% undersampling and the coarse problem has a unique solution \cite{Kuske2019}. As a consequence we quickly approach the boundary of the box. This results in an inefficient update using the Armijo line search and more similar curves for single-level resp.~two-level (2L) Riemannian gradient descent.}
  \label{fig:2LevelResults} 
  \end{figure}
  
  \begin{figure}[ht]
 \centering
   \includegraphics[width=0.45\textwidth]{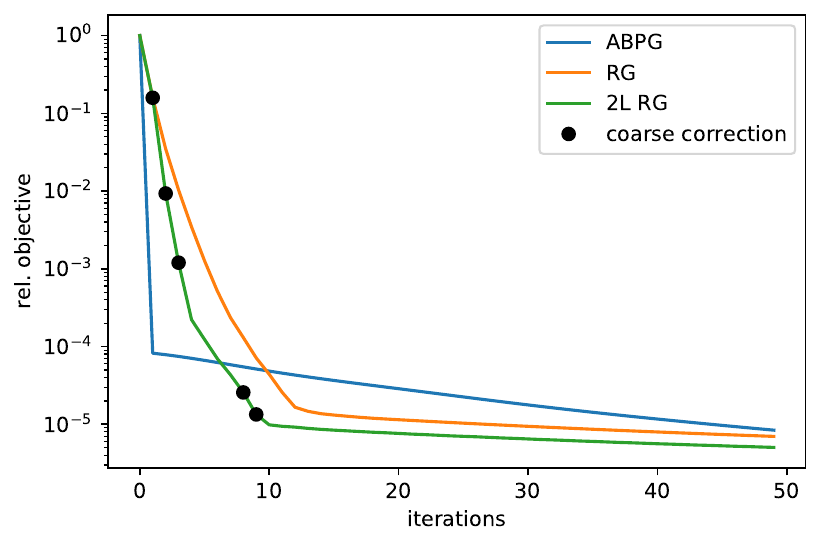}
   \includegraphics[width=0.45\textwidth]{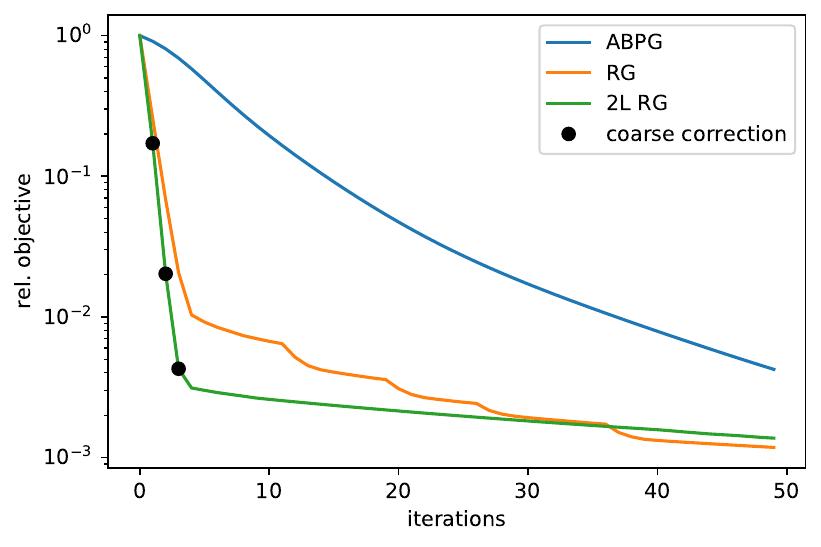}\\
 \includegraphics[width=0.45\textwidth]{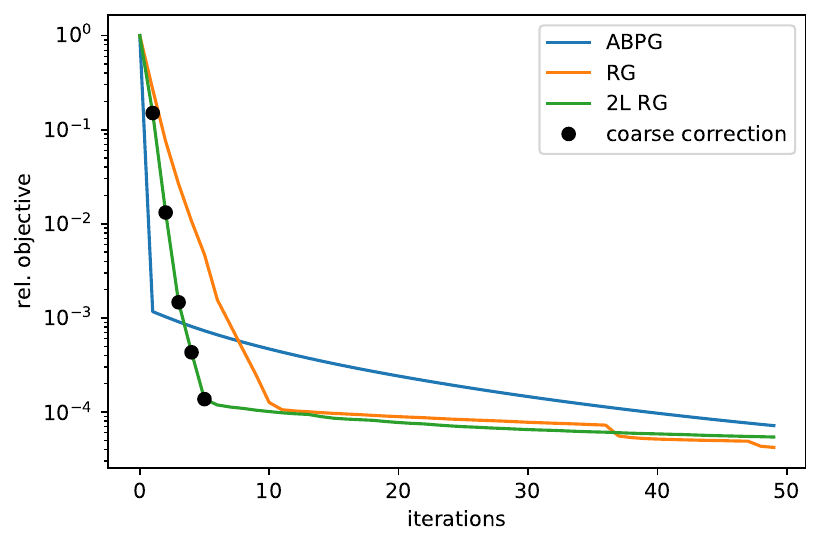}
 \includegraphics[width=0.45\textwidth]{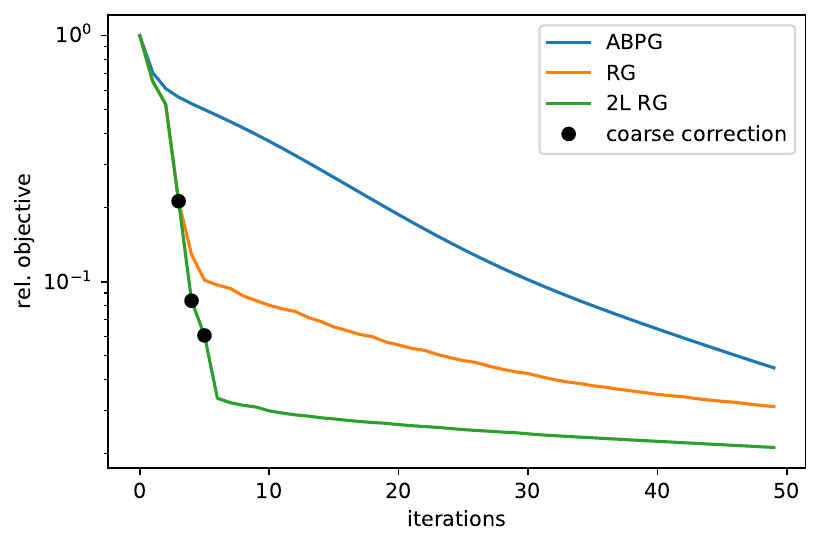}\\
 \includegraphics[width=0.45\textwidth]{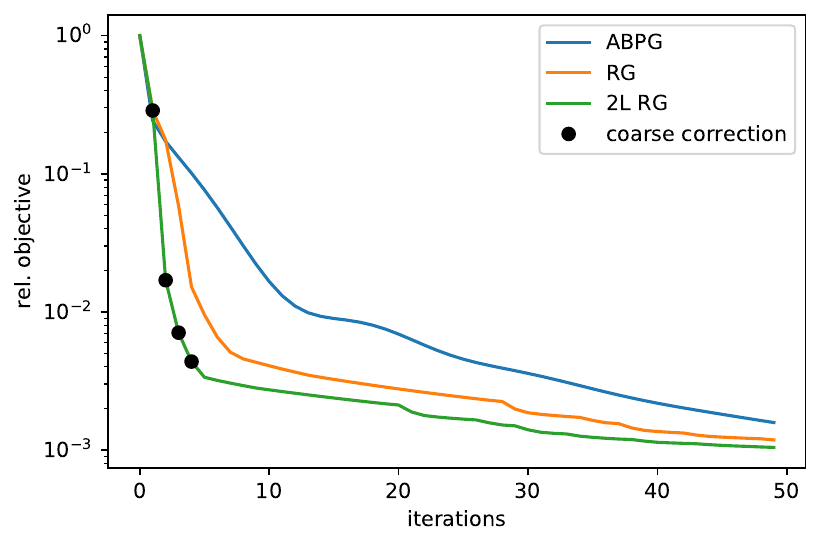}
 \includegraphics[width=0.45\textwidth]{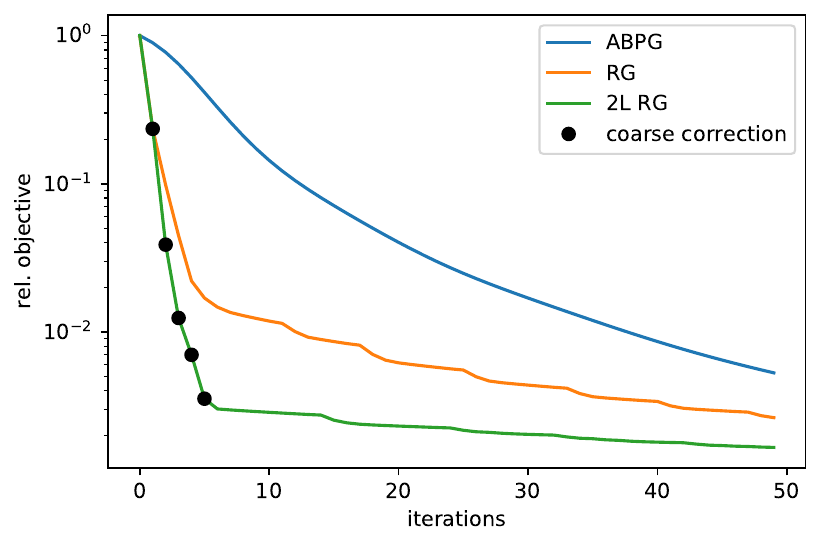}
  \caption{
  \textbf{2\% undersampling: comparison of relative objective function values} for single-level resp.~two-level (2L) Riemannian gradient descent (RG, 2L RG) and ABPG \cite{Hanzely:2021vc}. 
  The left column corresponds to the first three phantoms: \texttt{gear},  \texttt{bone}, \texttt{vessel} in 
  Figure \ref{fig:phantoms}, the right columns to the last three:  \texttt{batenburg}, \texttt{roux}, \texttt{skulls}.
  Black dots indicate when descent directions were computed on coarser grids. The two-level schemes - where we now have 4\% undersampling - aggressively decreases the objective, in particular for more homogeneous phantoms
 in the right column.
  }
  \label{fig:2LevelResults-20angles} 
\end{figure}

 \begin{figure}[ht]
 \centering
   \includegraphics[width=0.45\textwidth]{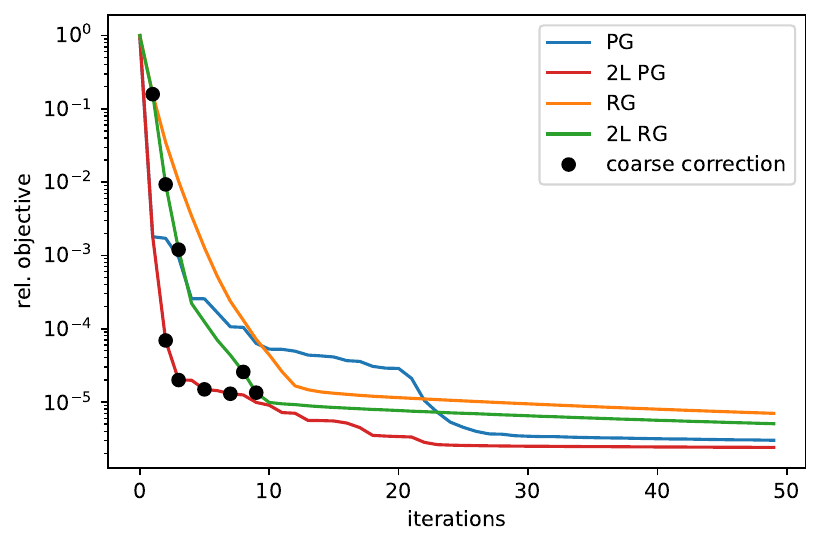}
   \includegraphics[width=0.45\textwidth]{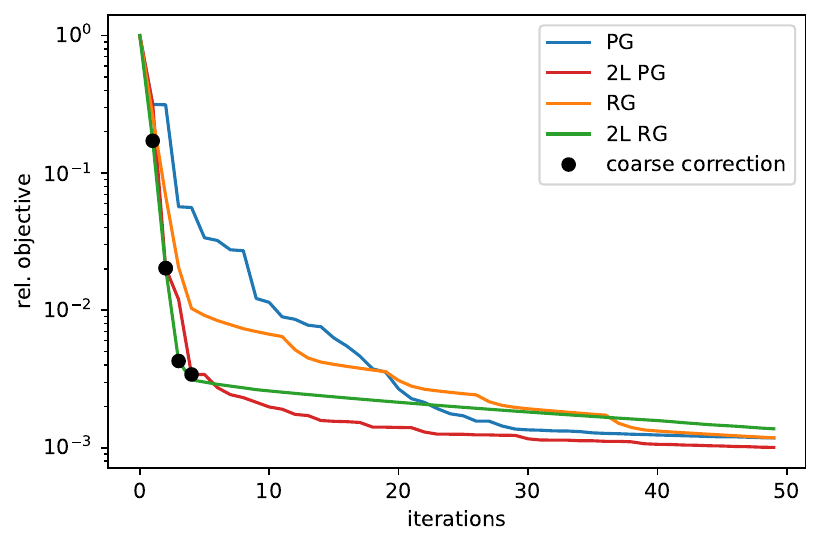}\\
 \includegraphics[width=0.45\textwidth]{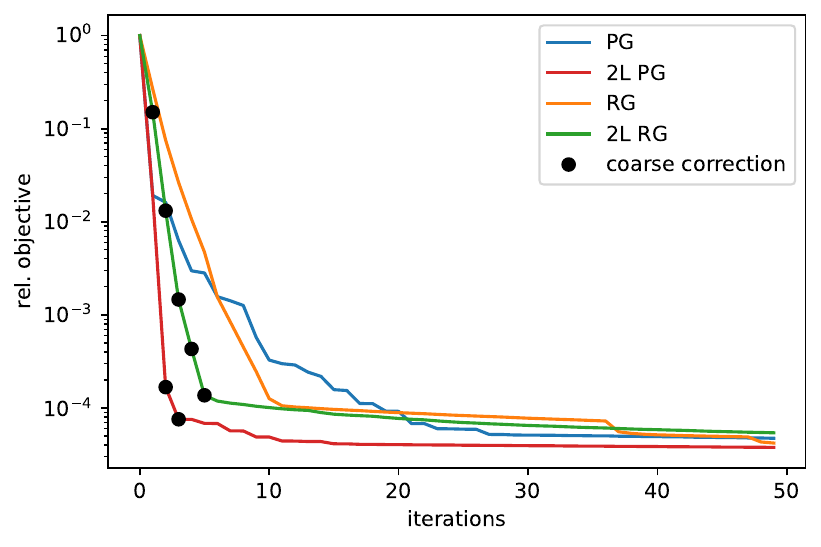}
 \includegraphics[width=0.45\textwidth]{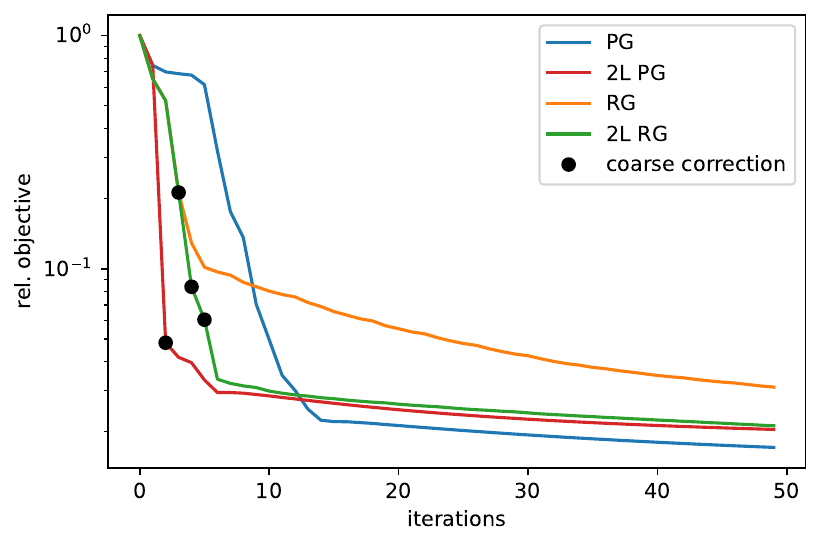}\\
 \includegraphics[width=0.45\textwidth]{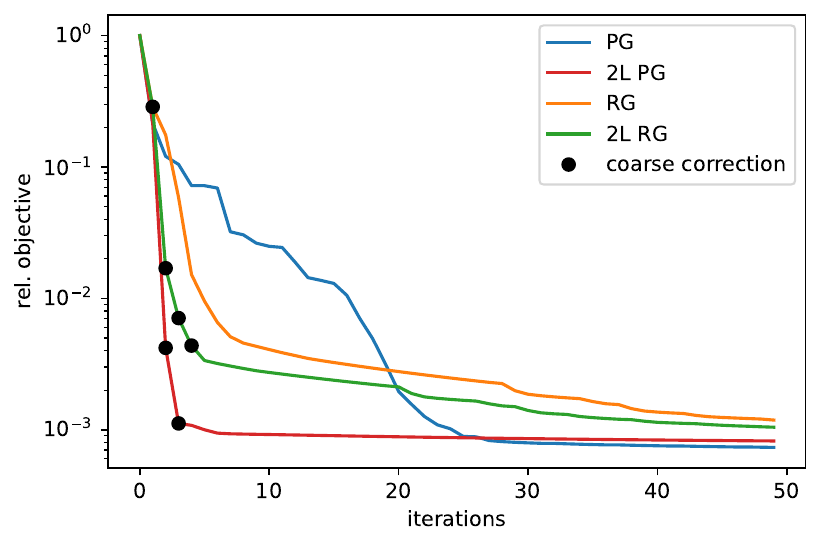}
 \includegraphics[width=0.45\textwidth]{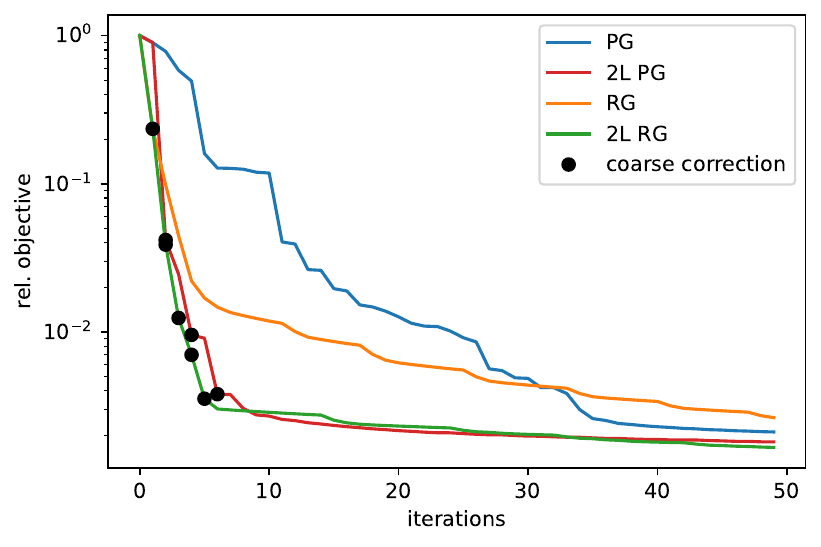}
  \caption{
  \textbf{Projected gradient (PG) vs. Riemannian gradient (RG) descent} in terms of relative objective values is compared for the single-level and two-level (2L) scenario with 2\% undersampled projection data. The left column corresponds to the first three phantoms: \texttt{gear}, \texttt{bone}, \texttt{vessel} in Figure \ref{fig:phantoms}, the right columns to the last three:  \texttt{batenburg}, \texttt{roux}, \texttt{skulls}. Black dots indicate when descent directions were computed on coarser grids. Overall, the Riemannian version of the gradient descent decreases faster the objective than projected gradient, since the descent direction already respects the constraints. Both two-level approaches show in addition more rapid reduction of the objective function.
  }
  \label{fig:2LevelResults-20anglesPGvsRG} 
\end{figure}


\section{Conclusions}\label{sec:conclusions}

We introduced a novel approach to geometric multilevel optimization. The approach employs information geometry in order to devise all ingredients of the iterative multilevel scheme. Invoking coarse level representations for computing descent directions effectively accelerates convergence. Experiments conducted for a range of instances of an ill-posed linear inverse problem with a non-quadratic convex regularizer and box constraints demonstrate promising performance of the method relative to the state-of-the-art.

The derivation of the approach for boxed-constrained convex programs can be transferred to other convex programs with simply structured feasible sets, analogous to turning parameter spaces of probability distributions into Riemannian manifolds in information geometry. Simplices instead of boxes as feasible sets provide an example \cite{Astrom:2017ac}.

This paper mainly focuses on the case of two-grid geometric optimization, which constitutes the core problem of multilevel optimization. Our future work will examine the multilevel case in detail, the selection of an appropriate number of resolution levels and related problems, and possible refinements. The latter includes machine learning components that can be controlled precisely, in order to optimize the transfer between levels, to adapt to each concrete problem instance and to given input data, for further accelerating the overall optimization process.

\section*{Acknowledgement}
We thank Jan Plier (Heidelberg University) for simulation code that efficiently evaluates
our objective. We greatly benefited form fruitful discussions with Christoph Schn\"orr (Heidelberg University). To Oana Curtef (University of W\"urzburg) we are indebted for her observation concerning line search. SM, SP and MZ gratefully acknowledge the generous and invaluable support of the Klaus Tschira Foundation.
\bibliographystyle{amsalpha}
\bibliography{main}
\end{document}